\documentclass[11pt,twoside,a4paper]{amsart}
\usepackage{amssymb}

\usepackage[english]{babel}
\usepackage{amsmath}
\usepackage{amsthm,amssymb,latexsym}
\usepackage[all,cmtip]{xy}
\usepackage{url,ae}
\usepackage{t1enc}
\usepackage{fancyheadings}
\usepackage{mathrsfs}
\usepackage{graphicx}
\usepackage{eufrak} 
\usepackage{geometry}

\date{\today}

\def\fff{{\sigma}}

\def\1{{\bf 1}}
\def\div{{\rm div}}
\def\ett{{\bf 1}}

\def\GZ{{\mathcal{GZ}}}
\def\Zy{{\mathcal Z}}

\def\deg{\text{deg}\,}

\def\w{\wedge}

\def\dbar{\bar\partial}

\def\C{{\mathbb C}}
\def\w{{\wedge}}
\def\P{{\mathbb P}}
\def\Pk{{\mathbb P}}

\def\A{{\mathcal A}}
\def\B{{\mathcal B}}

\def\S{{\mathcal S}}

\def\Cu{{\mathcal C}}

\def\Hom{{\rm Hom\, }}
\def\codim{{\rm codim\,}}

\def\Kers{{\mathcal Ker\,}}

\def\Z{{\mathbb Z}}
\def\E{{\mathcal E}}

\def\Re{{\rm Re\,  }}

\def\J{{\mathcal J}}
\def\nbh{neighborhood }

\def\be{\begin{equation}}
\def\ee{\end{equation}}

\def\Ok{\mathcal O}
\def\mult{{\rm mult}}

\def\s{{\circ}}

\newtheorem{thm}{Theorem}[section]
\newtheorem{lma}[thm]{Lemma}
\newtheorem{cor}[thm]{Corollary}
\newtheorem{prop}[thm]{Proposition}

\theoremstyle{definition}

\newtheorem{df}[thm]{Definition}

\theoremstyle{remark}

\newtheorem{preremark}[thm]{Remark}
\newtheorem{preex}[thm]{Example}

\newenvironment{remark}{\begin{preremark}}{\qed\end{preremark}}
\newenvironment{ex}{\begin{preex}}{\qed\end{preex}}

\numberwithin{equation}{section}

\title[Global representation of Segre numbers by Monge-Amp\` ere products]{Global representation of Segre numbers by Monge-Amp\` ere products}

\begin{document}

\date{\today}

\author[Andersson \& Eriksson \& Samuelsson Kalm \& Wulcan \& Yger]{Mats Andersson \& Dennis Eriksson \& H\aa kan Samuelsson Kalm \& \\
Elizabeth Wulcan \&  Alain Yger}

\address{Department of Mathematics\\Chalmers University of Technology and the University of
Gothenburg\\S-412 96 G\"OTEBORG\\SWEDEN}
\address{Institut de Math\'ematique \\
Université Bordeaux 1 \\
33405, Talence \\
France}

\email{matsa@chalmers.se, dener@chalmers.se, hasam@chalmers.se, wulcan@chalmers.se, Alain.Yger@math.u-bordeaux.fr}


\thanks{The first, third and fourth  author  were
  partially supported by the Swedish
  Research Council}

\begin{abstract}
On a reduced analytic space $X$ we introduce the concept of a generalized cycle, which
extends the notion of a formal sum of analytic subspaces to include also a form part.
We then consider a suitable equivalence relation and corresponding quotient $\mathcal{B}(X)$
that we think of as an analogue of the Chow group and a refinement of de Rham cohomology.
This group allows us to study both global and local intersection theoretic properties.

We provide many $\mathcal{B}$-analogues of classical intersection theoretic constructions:
For an analytic subspace $V\subset X$ we define a $\mathcal{B}$-Segre class, which is an element of
$\mathcal{B}(X)$ with support in $V$. It satisfies a global King formula and, in particular, its multiplicities
at each point coincide with the Segre numbers of $V$. When $V$ is cut out by a section of a vector bundle
we interpret this class as a Monge-Amp\`ere-type product. For regular embeddings we construct 
a $\mathcal{B}$-analogue of the Gysin morphism.
\end{abstract}


\maketitle

\section{Introduction}

Throughout this paper $X$ is a reduced analytic space of pure dimension $n$ and
 $\J\to X$ is  a coherent ideal sheaf  with zero set $Z$  with codimension
$\kappa$. 
Tworzewski, \cite{Twor}, and 
Gaffney and Gassler, \cite{GG},
independently introduced,
at each point $x\in X$,  numbers $e_{\kappa}(\J,x),\ldots, e_n(\J,x)$ that generalize
the Hilbert-Samuel multiplicity at $x$. 
These definitions,
although slightly different, are both of a geometric nature.
There is also a purely algebraic definition, see \cite{AM} and \cite{AR} by Achilles-Manaresi and Achilles-Rams, respectively. 
In \cite{aswy} were introduced semi-global currents
whose Lelong numbers at $x$ are precisely the $e_k(\J,x)$,  thus providing an analytic definition. 
Following \cite{GG} we call these numbers Segre numbers and, indeed, we will see in Theorem~\ref{genking} below that
they are closely related to Segre classes.

The main goal in this paper is to define concrete global analytic-geometric objects
that represent the Segre numbers at each point.
A secondary goal is to provide a framework, based on currents,  to  
connect local intersection theory with global constructions.

Intersection theory 
deals with the $\Z$-module $\Zy(X)$ of analytic cycles  and its  quotient 
module $\A(X)$, the Chow group. In general there are no cycles  or
elements in $\A(X)$ that can represent the Segre numbers at each point.  
To find global representations we introduce an extension $\GZ(X)$ of $\Zy(X)$ that we call the $\Z$-module of
\emph{generalized cycles}. Formally the elements in $\GZ(X)$ are 
a certain kind of closed currents but we prefer to think of them as geometric objects. In particular,
ordinary cycles are certainly geometric objects but formally represented by their associated Lelong currents in $\GZ(X)$.   
Many of the well-known geometric properties of
$\Zy(X)$ extend to $\GZ(X)$:  We have the natural grading by dimension
$\GZ(X)=\oplus_0^n\GZ_k(X)$,    where
$\GZ_k(X)$ are the submodules of generalized cycles of pure dimension $k$. At each point 
a generalized cycle $\mu$ has a well-defined multiplicity that is an integer.
There is a notion of Zariski support of $\mu$, 
and any $\mu$ has a unique decomposition in irreducible components.  
Moreover, $\GZ(X)$ is closed under multiplication by components
of Chern and  Segre forms of Hermitian vector bundles\footnote{All 
vector bundles in this paper are holomorphic.}. 
To get independence of various choices we introduce a certain quotient module $\B(X)$
of $\GZ(X)$; $\B(X)$ preserves the above-mentioned geometric properties of $\GZ(X)$. For instance,  $\Zy(X)$
is a submodule of $\B(X)$, we have a grading by dimension
$\B(X)=\oplus_0^n\B_k(X)$ 
and well-defined multiplicities, etc.  Moreover, $\B(X)$ admits a multiplication by components of Chern 
and Segre classes.
A proper mapping\footnote{Mappings 
between spaces are always assumed to be holomorphic.} $f\colon X'\to X$ induces a mapping 
$f_*\colon \GZ(X')\to \GZ(X)$, which in turn induces a mapping $\B(X')\to \B(X)$.
Assume that $i\colon V\hookrightarrow X$ is a subvariety.  
The image of the injective mapping $i_*\colon \GZ(V)\to \GZ(X)$ is precisely
the elements in $\GZ(X)$ that have Zariski support in $V$. 
Conceptually we identify $\GZ(V)$ with its image.  
In the same way  $\B(V)$ is identified with the elements in  $\B(X)$ that have Zariski support on $V$.

\smallskip
We define the {\it $\B$-Segre class}
$S(\J,X)$ in $\B(Z)$ in analogy with the Segre class in $\A(Z)$, cf.~Remark~\ref{asegre} below:
First assume that $X$ is irreducible. If $\J$ vanishes identically on $X$, then $S(\J,X)=1$ on $X$.
Otherwise,  let $\pi\colon X'\to X$ be any modification of $X$ such that
the ideal sheaf  $\pi^*\J$ is principal\footnote{In this paper we let $\pi^*\J$ denote the ideal generated by the pullback of generators of $\J$.},  let
$c_1(L)$ be the first Chern class of the line bundle $L$ defining the
exceptional divisor $D$ in $X'$, and let $[D]$ be its Lelong current. For instance, $X'$ can be
the blow-up of $X$ along $\J$. 
Then 
\begin{equation}\label{segredef}
S(\J,X)=\pi_*\big([D]\w \frac{1}{1+c_1(L)}\big)=
\sum_{j=1}^{n} (-1)^{j-1}\pi_*\big([D] \w c_1(L)^{j-1}\big).
\end{equation}
Since $\pi$ is proper, 
\eqref{segredef} defines
an element in $\B(Z)$.  We will see that it is independent of the choice of
modification. 
If $X$ consists of the irreducible components $X_1, X_2,\ldots $, then
we let $S(\J, X)=S(\J,X_1)+S(\J,X_2)+\cdots$ which is a locally finite sum on $X$.

We are now ready to formulate our first main result, which is a generalized King formula,
\cite{GriffKing, King}, for these objects and that in particular
provides the desired global representation of the Segre numbers of $\J$. 
Let $S_k(\J,X)$ be the component of $S(\J,X)$ in $\B_{n-k}(Z)$.

\begin{thm}[Global generalized King formula] \label{genking} 
Let $\J\to X$ be a coherent ideal  sheaf over a reduced  analytic space  of pure dimension $n$ 
and let $\kappa$ be the codimension of the zero set $Z$ of $\J$.
The class $S(\J,X)$ only depends on the integral closure class of $\J$.
We have unique decompositions 
\begin{equation}\label{king} 
S_{k}(\J,X)=\sum_j \beta_j^k [Z_j^k]+N_k^\J,  \quad k=0,1,2, \ldots,
\end{equation}
in $\B_{n-k}(Z)$, 
where  $Z_j^k$ are the (Fulton-MacPherson) distinguished varieties of $\J$ of codimension $k$,
$\beta_j^k$ are positive integers, and $N_k^\J$ has the following property:   The multiplicities
$\mult_x  N_k^\J$ 
are nonnegative integers, and the set of $x$ where $\mult_x  N_k^\J \ge 1$ has
codimension at least $k+1$.  
Moreover, $S_{k}(\J,X)=0$ for $k<\kappa$, $N^\J_{\kappa}=0$, and  
\begin{equation}
\mult_x S_k(\J,X)=e_{k}(\J, x), \quad k=\kappa, \ldots, n, \ \ x\in X.
\end{equation}
\end{thm}

Our next objective is to present specific representatives for the $\B$-Segre class
 $S(\J,X)$.   Assume that we have a holomorphic section $\fff$ of 
a Hermitian vector bundle $E\to X$ such that $\fff$ generates $\J$.  If $X$\   is projective
such a $\fff$ always exists.   One can give a meaning to the  Monge-Amp\`ere products
$
(dd^c\log|\fff|^2)^k
$
for all $k=0, 1,\ldots$, as follows.  To begin with it is defined as $1$ when $k=0$.  The higher powers
are defined recursively in \cite{A2} as 
\begin{equation}\label{recur}
(dd^c\log|\fff|^2)^k=dd^c\big(\log|\fff|^2 \1_{X\setminus Z}(dd^c\log|\fff|^2)^k\big).
\end{equation}
For $k\le \codim \J$ this definition coincides with 
Demailly's extension of the
classical Bedford-Taylor definition. 
Proposition~4.4 in \cite{A2} states that
\begin{equation}\label{44}
(dd^c\log|\fff|^2)^k=\lim_{\epsilon\to 0} (dd^c\log(|\fff|^2+\epsilon))^k,
\end{equation}
which gives further motivation for the notation. 
It was recently proved in \cite{abw} that one can also take the limit when
$\ell\to \infty$ of $(dd^c u_\ell)^k$, where  
$u_\ell=\max(\log|\fff|^2, -\ell)$;  several other, but not all (sic!),
sequences of plurisubharmonic  functions decreasing to $\log|\fff|^2$ also work.

 \begin{thm}\label{thm2}
Let $\fff$ be a holomorphic section of a Hermitian vector bundle $E\to X$
and let $\J$ be the ideal sheaf generated by $\fff$.  The current
\begin{equation}\label{orvar}
M_k^{\fff}:=\1_Z(dd^c\log|\fff|^2)^k,  \quad k=0,1,2, \ldots,
\end{equation}
is a generalized cycle that represents the $\B_{n-k}(Z)$-class $S_{k}(\J,X)$.
\end{thm}

Since $\Zy(Z)$ is a subgroup of $\B(Z)$ we conclude the following global version of \cite[Theorem~1.1]{aswy} 
from Theorems~\ref{genking} and \ref{thm2}.

\begin{cor}\label{snabela}
We have  unique decompositions
\begin{equation}\label{kingmf} 
M_k^\fff =\sum_j \beta_j^k [Z_j^k]+N_k^\fff,  \quad k=\kappa,  \ldots,  n,
\end{equation}
where $N_k^\fff$ are elements in $\GZ_{n-k}(Z)$.  
In particular,  $\mult_x M_k^\fff$ is equal to the Segre number
$e_{k}(\J,x)$ at each point $x$. 
\end{cor}

Given a generalized cycle $\mu\in\GZ_m(X)$ with Zariski support $|\mu|$ we define 
in Section~\ref{segreclass} for each $k\ge 0$ a generalized cycle
$M_k^{\fff}\w\mu$ with Zariski support on $Z\cap |\mu|$ and dimension $m-k$.
Its class in $\B_{m-k}(Z\cap|\mu|)$
only depends on $\J$ and the class of $\mu$ in $\B_m(X)$. 
We let $M^\fff\w\mu= M^\fff_0\w\mu+M^\fff_1\w\mu+\cdots$.  
We think of $M^\fff\wedge\mu$ as  (the push-forward to $X$ of) a representative of
the Segre class $S_k(\J,\mu)$  of $\J$ on $\mu$, cf.~Remark~\ref{tellus}.

\smallskip
Notice  that a coherent ideal sheaf $\J\to X$ can be identified with the, 
possibly non-reduced, embedded space $Z_\J\hookrightarrow X$ with
underlying reduced space $Z$ and structure sheaf $\Ok_X/\J$.  
If $i\colon \mu\hookrightarrow X$ is a reduced analytic subspace, then 
we denote by $s(\J,\mu)$ the class in  $\A(Z)$, called the Segre class,
that is denoted by $s(Z_{i^*\J},\mu)$ in \cite{Fult},
cf.~Remark~\ref{asegre} below.

\smallskip
In intersection theory the notion of regular embedding $W\hookrightarrow X$ 
plays a central role. With the identification above ``regular'' means that the associated sheaf $\J\to X$
is locally a complete intersection\footnote{We will assume that a regular embedding has codimension
$\kappa\ge 1$.}.
Since our second goal concerns intersection theory we will pay special attention to such sheaves $\J$
and describe $S(\J,X)$ in more detail. 
In this case the normal cone $N_\J X$ is a vector bundle over 
$Z$ and we let  
$s(N_\J X)=1+s_1(N_\J X)+s_2(N_\J X)+\cdots+s_{n-\kappa}(N_\J X)$  
be its associated total Segre class. 
Here lower index $\ell$ denotes the component of bidegree $(\ell,\ell)$.    
Let $[Z_\J]$ be (the Lelong current of) the fundamental
cycle of $Z_\J$, cf.~\cite[Ch.\ 1.5]{Fult}.

\begin{prop}\label{gata1}
If $\J$ defines a regular embedding $Z_\J\hookrightarrow X$,  then
$$
S_{k}(\J,X)=s_{k-\kappa}(N_\J X)\w [Z_\J], \quad k=\kappa, \ldots, n,
$$
in $\B_{n-k}(X)$.
\end{prop}

As in the case with general ideal sheaves we are interested in specific representatives,
so let us assume that $\J$ is defined by a section $\varphi$ of a Hermitian
vector bundle $F\to X$ and let
$F'$ be the pull-back of $F$ to $Z$. There is a canonical holomorphic embedding
$i_\varphi\colon N_\J X\to F'$ of $N_\J X$ in $F'$,
see Section~\ref{regular}.
Let us equip
$N_\J X$ with the induced Hermitian metric and let $\hat s(N_\J X)$ be 
the associated total Segre form which indeed  is smooth on $Z$,  see 
Section~\ref{prel}.

\begin{prop}\label{gata2}
If $\varphi$ is a section of the Hermitian vector bundle $F$ defining
$\J$, then we have the equality of generalized cycles  
$$
M^{\varphi}_{k} =\hat s_{k-\kappa}(N_\J X)\w [Z_\J], \quad k=\kappa, \ldots, n.
$$
\end{prop}

We have  a mapping
\begin{equation}\label{bgysin}
\B_k(X)\to \B_{k-\kappa}(Z),\quad 
\mu\mapsto \big(c(N_\J X)\w S(\J,\mu) \big)_{k-\kappa},
\end{equation}
where lower index denotes dimension,
and $c(N_\J X)=1/s(N_\J X)$ is the total Chern class of $N_\J X$.
If we choose a section $\varphi$ of $F$ as above we get a representing mapping 
\begin{equation}\label{jupiter}
\GZ_k(X)\to\GZ_{k-\kappa}(Z),\quad
\hat\mu\mapsto \big(\hat c(N_\J X)\w M^\varphi\w\hat \mu \big)_{k-\kappa},
\end{equation}
where $\hat c(N_\J X)$ is the associated total Chern form. 
The mapping \eqref{bgysin} is a  $\B$-analogue of the Gysin
mapping,   \cite[Proposition~6.1]{Fult}, see Section~\ref{prel} for the notation,
\begin{equation}\label{gysin}
\A_k(X)\to \A_{k-\kappa}(Z),\quad 
\mu\mapsto \big(c(N_\J X) \cap s(\J, \mu)\big)_{k-\kappa}.
\end{equation}

\smallskip
In Section~\ref{faksimil} we introduce a quotient space
$\widehat H^{\ell,\ell}(X)$ of closed $(\ell,\ell)$-currents with support on $X$,
coinciding with the usual de~Rham cohomology in case $X$  is smooth.
There are natural mappings $\A_k(X)\to \widehat H^{n-k,n-k}(X)$ 
and $\B_k(X)\to \widehat H^{n-k,n-k}(X)$. 

\begin{prop}\label{ab1}  
For each $k$, the  images of $\A_k(X)$ and $\B_k(X)$  in $ \widehat H^{n-k,n-k}(X)$
coincide. 
\end{prop}

\begin{prop}\label{ab0}
Assume that $\J\to X$ defines a regular embedding $Z_\J\hookrightarrow X$
of codimension $\kappa$ and let $Z$ be the (reduced) zero set of $\J$.  
If $\mu$ is a cycle on $X$, then the images in 
$\widehat H^{*,*}(Z)$ of the Gysin and the $\B$-Gysin mappings, \eqref{gysin} and \eqref{bgysin}, respectively,
of $\mu$ coincide.
\end{prop}

In Section~\ref{vogelsec} we consider a general ideal sheaf $\J\to X$ that is generated by 
a tuple $\fff=(\fff_0,\ldots, \fff_m)$ of global sections
 of a  
line bundle $L\to X$.  In this situation St\"uckrad-Vogel, \cite{SV}, introduced an algorithm to produce concrete
cycles, St\"uckrad-Vogel  cycles, that determine a  Chow class $v(\J,L,X)$, which is related to  $s(\J,X)$ via
van~Gastel's formulas, \cite{gast}.  Given a Hermitian metric on $L$ we define a global  
generalized cycle $M^{L,\fff}$ by taking a certain mean
value of  St\"uckrad-Vogel cycles.  If we consider $\fff$ as a section of
$E=\oplus_0^m L$  we have an analogue of van~Gastel's
formulas relating $M^{L,\fff}$ and $M^{\fff}$ as elements in $\GZ(X)$.

\section{Preliminaries}\label{prel}

Locally there is an embedding
$i\colon X\to \Omega\subset\C^N$ into an open subset $\Omega\subset\C^N$. 
The sheaf  $\E^{n-\ell,n-k}_X$ of smooth  $(n-\ell,n-k)$-forms on $X$ is by definition
the quotient sheaf
$\E^{n-\ell,n-k}_{\Omega}/\Kers i^*$, where $\Kers i^*$ is the sheaf of forms $\xi$ on $\Omega$ such that
$i^*\xi$ vanish on  $X_{\text{reg}}$. Since all embeddings are essentially equivalent, this definition
is independent of the choice of embedding. 
The sheaf
$\Cu_X^{\ell,k}$ of  currents of bidegree $(\ell,k)$ on $X$ is by definition the dual
of  $\E^{n-\ell,n-k}_X$.  Given the embedding $X\to \Omega$, currents $\mu$ in
$\Cu^{\ell,k}_X$ can be identified with
currents $\mu'=i_*\mu$ on $\Omega$ of bidegree $(N-n+\ell,N-n+k)$  that
vanish on $\Kers i^*$.
We say that $\mu$ has order zero if $i_*\mu$ has order zero; recall that this means that $i_*\mu$
has measure coefficients.
A current $\mu$ in $\Cu^{n-d, n-d}_X$ is said to have {\it (complex) dimension} $d$.
If $f\colon X\to X'$
is  proper,  then $f^*$ is well-defined on smooth forms and $f_*$ is well-defined on currents
and preserves dimension, see \cite{ASK}.
If $\mu$ is a current on $X$ and $\eta$ is a smooth form on $X'$, then 
\begin{equation}\label{ingen}
\eta\wedge f_*\mu = f_*(f^*\eta\wedge \mu). 
\end{equation} 
Moreover, if $\mu$ has order zero then so has  $f_*\mu$ 
and  
\begin{equation}\label{linalg}
\1_V f_*\mu=f_*(\1_{f^{-1} V} \mu),
\end{equation}
where $\1_V$ is the characteristic function of the analytic subset $V$. 
If $\mu$ is a closed positive current then so is $f_*\mu$. 
The \emph{Lelong number} $\ell_x\mu$ of $\mu$ at
$x$ is defined as the Lelong number
of $i_*\mu$ at $i(x)$ where $i$ is a local embedding in a smooth
manifold, see, e.g.,  \cite[Section~2.2]{aswy}.
If $V$ is a subvariety of $X$ of pure dimension $d\ge 0$, then
there is an associated  closed positive current of
dimension $d$, the {\it Lelong current},
$$
\phi\mapsto [V].\phi=\int_{V_{\rm reg}}\phi.
$$

\smallskip

Recall that to any Hermitian line bundle
$L\to X$ there is an
associated (total) Chern form $\hat c(L)=1+\hat c_1(L)$.
If $L'$ is the same line bundle but with another Hermitian metric, then
there is smooth function $\xi$ on $X$ such that
\begin{equation}\label{spjut}
\hat c_1(L')-\hat c_1(L)=dd^c\xi.
\end{equation}
Assume that $E\to X$ is a Hermitian vector bundle of rank $r$, and let
$\pi\colon \P(E)\to X$ be the projectivization of $E$, by which we mean the 
projective bundle of lines through the origin in $E$.
 Let $L=\Ok(-1)\subset \pi^*E$ be the tautological
line bundle equipped with the induced Hermitian metric, and let $\hat c(L)$ be its 
Chern form. The (total) {\it Segre form} of $E$ is  defined as
\begin{equation}\label{orgie}
\hat{s}(E)=\pi_*(1/\hat c(L)).
\end{equation}
Thus\footnote{It is not obvious  that $\hat s_0(E)=1$; however it follows from the corresponding
statement for the Chow class, see \cite{Fult},  or from \eqref{mor} below.} 
$\hat{s}(E)=1+\hat s_1(E)+\hat s_2(E)+\cdots $
where 
\begin{equation}\label{skot}
\hat s_\ell(E)=(-1)^{\ell+r-1}\pi_* \hat c_1(L)^{\ell+r-1}
\end{equation}
is the component of bidegree $(\ell,\ell)$.
It is indeed is a smooth form on $X$:  if $X$ is smooth this follows since 
$\pi$ is a submersion and in general it follows by embedding
$X$ locally  in a smooth 
space and extending $E$ to a Hermitian bundle over this space.

Let $X'$ be another analytic space and $f \colon X'\to X$
a proper mapping. Then the tautological line bundle 
$L'\to\P(f^*E)$ 
associated with
$\P(f^*E)\to X'$ is the pullback  of $L\to \P(E)$ under the induced map $\tilde{f}\colon \P(f^*E)\to \P(E)$
and so $\hat c_1(\tilde{f}^*L)=\tilde{f}^*\hat c_1(L)$.
It follows that 
 \begin{equation}\label{orgie1}
\hat s_k( f^*E)= f^*\hat s_k(E).
\end{equation}
If $E$ is a line bundle, then $\P(E)=X$, $L=E$, and hence
\begin{equation}\label{smal}
\hat c(E)=1/\hat s(E).
\end{equation}
For a general Hermitian vector bundle $E\to X$ we take \eqref{smal} as the definition
of its (total) Chern form.  Thus $\hat c(E)=1+\hat c_1(E)+\hat c_2(E)+\cdots$ where the component $\hat c_k(E)$
of bidegree $(k,k)$ is a polynomial in the $\hat s_\ell(E)$.   
From \eqref{orgie1} we get
\begin{equation}\label{orgie2}
\hat c_k(f^*E)=f^*\hat c_k(E).
\end{equation}

Let $E$ and $E'$ be the same bundle but with two different Hermitian metrics and let
$L$ and $L'$ be the associated Hermitian line bundles over $\P(E)$.
In view of \eqref{spjut},  \eqref{skot} and \eqref{smal}   (and that $\pi$ is a submersion)
we have, for $k\ge 1$, that
\begin{equation}\label{prom}
 \hat s_k(E')-\hat s_k(E)=dd^c\omega_s,  \quad \hat c_k(E')-\hat c_k(E)=dd^c\omega_c,
\end{equation}
for suitable smooth $(k-1,k-1)$-forms $\omega_s,\omega_c$ on $X$.
We let $s_k(E)$ and $c_k(E)$ denote the cohomology classes, which we for simplicity refer to as
the Segre and Chern classes,  although we only consider representatives obtained
from a Hermitian metric as above.

The Hermitian metric on $E$ determines
a Chern connection and thus a curvature tensor $\Theta_E$. 
It is proved in \cite[Proposition~6]{Mor} that the definition used here and the differential-geometric
definition of Chern form coincide, that is, 
\begin{equation}\label{mor}
\hat c(E)=\det(I+(i/2\pi)\Theta_E).
\end{equation}

An analytic $k$-cycle $\mu$ on $X$ is a formal locally\footnote{Algebraic geometry only deals with finite
linear combinations, but we use the more ``analytic'' definition.} 
finite linear combination $\sum
a_j V_j$, where $a_j\in \Z$ and  $V_j\subset X$ are irreducible
analytic sets of dimension $k$. We let 
\[
[\mu]:=\sum a_j [V_j]
\]
be its associated Lelong current.
Note that if $V_j$ has dimension $n$ (the dimension of $X$),
then $[V_j]=\1_{V_j}$. 
We will denote the $\Z$-module of analytic $k$-cycles on $X$ by
$\Zy_k(X)$. 
The \emph{support} $|\mu|$ of the cycle $\mu$ is defined as the union of the $V_j$
for which $a_j\neq 0$ and it coincides with the support of the current $[\mu]$. 
Recall that 
\begin{equation}\label{imam}
\mult_x \mu=\ell_x [\mu], 
\end{equation}
where $\ell_x\gamma$ denotes the \emph{Lelong number} of the closed positive current $\gamma$ at
$x$, and $\mult_x\mu$ is the \emph{multiplicity} of $\mu$ at
$x$ (defined as in \cite[Ch.~2.11.1]{Ch}), see, e.g.,
\cite[3.15, Proposition~2]{Ch}.

Let $f\colon X'\to X$ be a proper mapping. 
For each irreducible subvariety $V\subset X'$, let $\deg f_V$
denote the degree of $f|_V\colon V\to f(V)$; if $\dim f(V)< \dim V$ it is
defined as zero. 
The \emph{push-forward} of $\mu\in\Zy_k(X')$ is the cycle
\begin{equation}\label{justnu}
f_*\mu =  \sum a_j \deg f_{V_j}  f(V_j), 
\end{equation}
in $\Zy_k(X)$, see, e.g., 
\cite[Section~1.4]{Fult}.
Since $f_*[V]=\deg f_V[f(V)]$ 
it  follows that 
\begin{equation}\label{rapport}
f_*[\mu]=[f_*\mu].
\end{equation}
In particular, if $i\colon X\to Y$ is an embedding in another reduced space $Y$,
then $\mu\in \mathcal{Z}_k(X)$ can be regarded as
a cycle on $Y$ and $i_*[\mu]=[\mu]$.
For the rest of this paper we often skip the notation $[\mu]$ and identify a cycle with its Lelong current.
\smallskip

Let  
$d^c=(\partial-\dbar)/4i\pi$
so that\footnote{We write $[0]$ 
rather than  $[\{0\}]$ for the point mass 
at $0$.} $dd^c\log|z|^2=[0]$
 in $\C$.
The Poincar\'e-Lelong formula, usually stated on a smooth manifold, 
has an extension to our nonsmooth case (see also Section~\ref{plvariants}).
We say that a meromorphic section of a line bundle is {\it non-trivial} if it is 
generically holomorphic and non-vanishing.

\begin{prop}[The Poincar\'e-Lelong formula]\label{plf}
Let $h$ be a non-trivial meromorphic 
section of a Hermitian line bundle $L\to X$.
Then $\log|h|^2$ has order zero on $X$,  
\begin{equation}\label{orgie3}
dd^c\log|h|^2=\lim_{\epsilon\to 0} dd^c\log(|h|^2+\epsilon)
\end{equation}
where $h$ is holomorphic, 
and there is a cycle $\div h$ such that
\begin{equation}\label{pl}
dd^c\log|h|^2=[\div h]-\hat c_1(L).
\end{equation}
\end{prop}

In case $X$ is smooth, $\div h$ is the usual divisor
defined by $h$.

\begin{proof} 
Let $\pi\colon X'\to X$ be a smooth modification. Since $\pi^*h$ is non-trivial on $X'$, 
$\log|\pi^* h|^2$ is locally integrable and hence a current of order $0$.
Since $\pi$ is a biholomorphism generically,
$\log|h|^2=\pi_*\log|\pi^*h|^2$.  Thus $\log|h|^2$ has order zero. 
For the same reason the limit \eqref{orgie3} holds where $h$ is holomorphic,
and   $\pi_*\hat c_1(\pi^*L)=\hat c_1(L)$.
By the Poincar\'e-Lelong formula on a smooth manifold,  
$
dd^c\log|\pi^*h|^2=[\div \pi^*h]-\hat c_1(\pi^*L).
$
Applying $\pi_*$ we get \eqref{pl} with $[\div h]=\pi_*[\div\pi^* h]$.
It follows from \eqref{rapport} that  $\div h$ is a cycle, and it follows from
\eqref{pl} that it is independent of the choice of modification. 
\end{proof}

Let  
$i\colon V\hookrightarrow X$ be a subvariety.  If $i^*h$ is non-trivial, 
then we say that $\div h$ {\it intersects $V$ properly}, and 
we have  the  {\it proper intersection} 
$[\div h]\w [V]:=i_* (\div i^*h)$,
cf.~\cite[Ch~2, 12.3]{Ch} and Section~\ref{plvariants} below.  
Letting  $\log|h|^2[V]:=i_*\log|i^*h|^2$ and noting that 
$\hat c_1(L)\w[V]=i_* \hat c_1(i^*L)$, 
we get from  \eqref{pl} the formula
\begin{equation}\label{pl2}
dd^c(\log|h|^2[V])=[\div h]\w[V]-\hat c_1(L)\w[V].
\end{equation}

Recall that $\mu\in \mathcal{Z}_k(X)$ is \emph{rationally equivalent} to $0$,
$\mu\sim 0$, if
there are  subvarieties
$i_j:W_j\hookrightarrow X$ of dimension $k+1$ and meromorphic 
non-trivial functions $g_j$ on $W_j$, 
such that, writing $g_j$ rather than $i_j^*g_j$ for simplicity, 
\begin{equation}\label{chow}
\mu=\sum_j (i_j)_* [\div g_j]=\sum_j (i_j)_*dd^c\log|g_j|^2=
\sum_j dd^c(\log|g_j|^2[W_j]),
\end{equation}
cf.~\eqref{pl2}, where the sums are locally finite. 
We denote the \emph{Chow group} of cycles $\mathcal{Z}_k(X)$ modulo rational
equivalence by $\A_k(X)$, cf.\  \cite[Chapter~1.3]{Fult}.
Note that if $X$ is irreducible and compact and $\mu$ is a Cartier divisor, then
$\mu\sim 0$ 
precisely if
$\mu=[\div g]= dd^c\log|g|^2$ for some meromorphic function $g$
on $X$, i.e., the line bundle
$\Ok(\mu)$ defined by $\mu$ is trivial.
Thus for Cartier divisors (when $X$ is compact), rational equivalence  precisely means linear equivalence.
If $f\colon X'\to X$ is a proper mapping and $\mu\sim 0$ in $\Zy_k(X')$, then
$f_*\mu\sim 0$ and thus \eqref{justnu} induces a mapping, cf.\ \cite[Theorem~1.4]{Fult},
\begin{equation}\label{kenny}
f_*\colon \A_k(X')\to \A_k(X).
\end{equation}

\smallskip
Each component $c_k(E)$ of a Chern class on $X$ induces a mapping
$\A_*(X)\to \A_{*-k}(X)$, $\mu\mapsto c_k(E)\cap \mu$,
see, \cite[Section~3.2]{Fult}.  If $h$ is a nontrivial
meromorphic section on $|\mu|$ 
of a line bundle $L$, then $c_1(L)\cap \mu$ is the class in $\A(|\mu|)$ defined
by $[\div h]\w\mu$.

\section{Generalized cycles}\label{gencykler} 
The generalized cycles is the smallest class of currents that is closed under proper direct images and 
contains sums of wedge products of Lelong currents
and components of Chern forms. More formally,
we say that a current $\mu$ in $X$ is
a {\it generalized cycle} if it is a locally finite 
linear combination over $\Z$ 
of currents of the form
$
\tau_* \alpha,
$
where $\tau\colon W\to X$ is a proper map, $W$ is smooth, and
$\alpha$ is a product of  components of  Chern forms for various Hermitian
vector bundles $E_j$ over $W$, 
i.e.,  
\begin{equation}\label{sedan}
\alpha=\hat c_{k_1}(E_1)\wedge \cdots \wedge \hat c_{k_r}(E_r).
\end{equation} 
We will keep this notation throughout this section.
Since we can restrict $\tau$ to each connected component of $W$ we can assume that 
$W$ is connected.

Note that a generalized cycle is a real current of order zero that is
closed (in particular it is normal) with
components of bidegree $(*,*)$. 
We let $\GZ_k(X)$ denote the $\Z$-module of such currents of {\it (complex) dimension} $k$,
i.e., of bidegree $(n-k,n-k)$, and let $\GZ(X)=\bigoplus
\GZ_k(X)$. 
If $\mu\in\GZ(X)$ and $\gamma$ is a component of a Chern form on $X$, then
$\gamma\w\mu\in\GZ(X)$. In fact, if $\mu=\tau_*\alpha$, where $\tau\colon W\to X$, 
then $\gamma\w\mu =\tau_*(\tau^*\gamma\w\alpha)$, cf.\ \eqref{ingen}.

\begin{remark}\label{kola}
In view of \eqref{smal}  each form \eqref{sedan} is a finite sum of similar forms but with
$\hat c$ replaced by $\hat s$.  Morover, we can assume that each factor in 
\eqref{sedan} is the first Chern form of a Hermitian line bundle. 
To see this it is enough to verify that any
$\alpha=\hat s_{k_1}(E_1)\wedge \cdots \wedge \hat s_{k_t}(E_t)$, where $E_j\to W$ are Hermitian vector bundles of rank $r_j$, 
is of this form. Let $\pi: W'\to W$ be the fiber product 
$
W'=\P(E_1)\times_W\cdots \times_W \P(E_t), 
$
let $L_j$ be the pullback to $W'$ of the tautological bundle
$\Ok(-1)\to \P(E_j)$, and let $\hat c_1(L_j)$ be the first Chern form
on $L_j$ induced by the metric on $E_j$. Then, cf.\ \eqref{skot}, 
\[
\alpha=\pm\pi_*\big ( \hat c_1(L_1)^{k_1+r_1-1}\wedge \cdots \wedge
\hat c_1(L_t)^{k_t+r_t-1}\big ).
\]
\end{remark}

\begin{lma}\label{rest}
Let $i\colon V\hookrightarrow X$ be a subvariety and $\mu\in\GZ(X)$. 

\smallskip
\noindent (i)\  Then $\1_V\mu\in\GZ(X)$. 
\smallskip

\smallskip
\noindent (ii)\
If 
\begin{equation}\label{husmus}
\mu=\sum_k(\tau_k)_* \alpha_k,
\end{equation}
where $\tau_k:W_k\to X$ are proper, $W_k$ are smooth and connected, and $\alpha_k$ are as in
\eqref{sedan}, 
then
\begin{equation}\label{bakom2}
\1_V \mu =\sum_{\tau_k(W_k)\subset V}
(\tau_k)_*\alpha_k.
\end{equation}
\end{lma}

\begin{proof}
Since the right hand side of \eqref{bakom2} is in
$\GZ(X)$ by definition, $(i)$ follows from $(ii)$.  
Assume now that \eqref{husmus} holds. 
By \eqref{linalg},
\begin{equation}\label{bakomliggande}
\1_V \mu =\sum_k (\tau_k)_*\big(\1_{\tau_k^{-1}V} \alpha_k\big
). 
\end{equation}
Assume that $\tau_k(W_k)\not \subset V$. Then 
$\tau_k^{-1} (V)$ has positive codimension in $W_k$ since $W_k$ is
connected. Since $\alpha_k$ is smooth it follows that 
$\1_{\tau_k^{-1}V} \alpha_k=0$, and hence the corresponding term in
\eqref{bakomliggande}, vanishes.
Thus \eqref{bakom2} holds.
\end{proof}

If $i\colon V\hookrightarrow X$ is a
subvariety of $X$, then  $[V]=i_*\alpha$, where $\alpha=1$, which is
the $0$th Chern form of any vector bundle over $V$. 
Thus we have an embedding
$$
\Zy_k(X)\to \GZ_k(X)
$$
and we think of $\Zy_k(X)$ as a subset of $\GZ_k(X)$.
If $h\colon X'\to X$ is  proper, then we have a natural mapping
\begin{equation}\label{venus}
h_*\colon \GZ_k(X')\to \GZ_k(X).
\end{equation} 
Indeed, if $\mu=\tau_* \alpha$, then $h_*\mu=(h\circ\tau)_*\alpha$ and
$h\circ\tau$ is proper.
In particular, if $i\colon V\hookrightarrow  X$ is a subvariety of $X$, then we have
an injective mapping  
\begin{equation}\label{snara}
i_*\colon \GZ_k(V)\to \GZ_k(X).
\end{equation}

Given $\mu\in\GZ(X)$ there  is a smallest variety $|\mu|$,
that we call
the {\it Zariski support} of $\mu$, such that
$\mu$ vanishes outside $|\mu|$. In fact, $|\mu|$ is the
Zariski closure of the support of $\mu$ as a current.

\begin{ex}\label{flyga} 
Assume that $X$ is irreducible and let $L\to X$ be the trivial
line bundle. Then any smooth function $\varphi$ on $X$ determines a
metric $|s|^2_L=|s|^2 e^{-\varphi}$ on $L$ with the corresponding
first Chern form $dd^c\varphi$. Since
$\mu:=dd^c\varphi$ can vanish on an open subset of $X$ without vanishing
identically,  it  is a non-zero generalized cycle with support
strictly smaller than $X$ but with $|\mu|=X$. 
\end{ex}

\begin{prop}[Dimension principle]\label{kvick}
Assume that  $\mu\in\GZ_k(X)$ has Zariski support $V$. 
If $\dim V =k$, then  $\mu\in\Zy_k(X)$.
If $\dim V<k$, then $\mu=0$.
\end{prop}

\begin{proof} Since $\mu$ is closed, of  dimension $k$
and order zero it follows from 
\cite[Corollary~III.2.14]{Dem2} that it is a sum of various
currents $a_j[V_j]$ where $V_j$ is irreducible of dimension $k$
and $a_j$ is a number.
By Proposition~\ref{gatsten} below the Lelong number of a generalized cycle is an integer at each point, and it
follows that the $a_j$ are integers.
If $\dim V<k$ it follows from  \cite[Thm~III.2.10]{Dem2} 
that $\mu=0$.
\end{proof}

\begin{ex}\label{lokalkonstant}
If $\mu\in\GZ_n(X)$,  then $\mu=\sum_j a_j\1_{X_j}$,
where $X_j$ are the irreducible components of $X$ and $a_j$ are integers. 
\end{ex}

\begin{prop}\label{basic}
The image of \eqref{snara} is precisely those $\mu\in\GZ_k(X)$ such that
$|\mu|\subset V$.
\end{prop}

Thus we can, and will indeed do,
identify generalized cycles on $V$ with generalized cycles in $X$
with Zariski support on $V$.

\begin{proof}
Assume that $\mu$ is on the form \eqref{husmus} and has support on
$V$.  Since  $\mu=\1_V\mu$ it follows from Lemma~\ref{rest} that
$\mu$ is equal to the right hand side of \eqref{bakom2}. 
For each of these $\tau_k$ we have a factorization
$\tau_k=i\circ \tau_k'$ where $\tau_k'\colon W_k\to V$ is proper.  
It follows that
$$
\mu':=\sum_k(\tau'_k)_* \alpha_k
$$
is in $\GZ(V)$ and $\mu=i_*\mu'$. 
\end{proof}

\begin{df} We say that  $\mu\in\GZ(X)$ is \emph{irreducible} in $X$ if
$|\mu|$  is irreducible and   
$\1_V\mu=0$ for any proper subvariety $V\subset |\mu|$. 
\end{df}

Thus irreducibility is connected to an irreducible subvariety of $X$.   If
$\mu\in\GZ(X)$ is irreducible with Zariski support $V$ it has a unique decomposition
\begin{equation}\label{hund}
\mu=\mu^p+\cdots+\mu^0,
\end{equation}
where $\mu^k$ is the component of dimension $k$ and $p=\dim V$.  It follows from  
Proposition~\ref{kvick}  that $\mu^p$ is $a[V]$ for some integer $a$.

\begin{lma}\label{gastub}
Assume that $\mu\in\GZ(X)$ is of the form $\mu=\tau_*\alpha$, where $\tau\colon W\to X$, $W$ is connected, and $\tau(W)=V$. 
Then $\mu$ is irreducible and $|\mu|=V$ or $\mu=0$.
\end{lma}

\begin{proof}
Since $W$ is irreducible, so is $V$. 
Clearly, $|\mu|\subset V$. 
Assume that $V'$ is a proper subvariety of $V$. Then $\tau^{-1}V'$ has
positive codimension in $W$ since $W$ is connected. Thus 
\begin{equation}\label{piano}
\1_{V'}\mu=\tau_*(\1_{\tau^{-1}V'}\alpha)=0
\end{equation}
since $\alpha$ is smooth.  If $|\mu|$ is a proper subvariety of $V$, therefore
$\mu=\1_{|\mu|}\mu=0$.  
If not, it follows from 
\eqref{piano} that $\mu$ is irreducible.  
\end{proof}

Notice that if $\mu,\mu'$ are irreducible with the Zariski support $V$, then 
$\mu+\mu'$ either vanishes or is again irreducible with Zariski support $V$.

\begin{prop}\label{goltupp}
Each   $\mu\in\GZ(X)$ has a unique decomposition
\begin{equation}\label{decomp}
\mu =\sum_j\mu_j,
\end{equation}
where  $\mu_j\in\GZ(X)$ are  irreducible with different Zariski supports.
\end{prop}

\begin{proof} 
We first prove the uniqueness. Let $V_j=|\mu_j|$. Assume that
\eqref{decomp} holds with $\mu=0$. If there are non-vanishing $\mu_j$ then we can choose
$k$ such that $\mu_k\neq 0$ and $V_k$ has minimal dimension among the
$V_j$ for which $\mu_j\neq 0$.  For each $j\neq k$ then
$V_k\cap V_j$ has positive codimension in $V_j$ and hence
$\1_{V_k}\mu_j=\1_{V_k\cap V_j}\mu_j=0$ since $\mu_j$ is irreducible.
Thus
$
\mu_k=\1_{V_k}\mu_k=\1_{V_k}\mu=0
$
which is a contradiction. We conclude that $\mu_j=0$ for all $j$.

To prove the existence, we may assume that $\mu$ is of the form 
\eqref{husmus},
where $\tau_k:W_k\to X$ and $W_k$ are connected.  
For each subvariety $V_j\subset X$ that appears as the Zariski support of one
of the summands in \eqref{husmus}, let $\mu_j=\sum(\tau_k)_*
\alpha_k$, where the sum is over all $k$
such that $\tau_k(W_k)=V_j$. Then,  by Lemma \ref{gastub},  $\mu_j$ is
irreducible with Zariski support $V_j$ or $\mu_j=0$. 
We now  get the decomposition \eqref{decomp}.  
\end{proof}

\begin{remark}\label{planta}
It follows from the proof that an irreducible $\mu\in\GZ(X)$ with $|\mu|=V$ is a finite sum
of terms like $i_*\tau_*\alpha$ where $\tau\colon W\to V$ is proper, $\tau(W)=V$ and  $W$ is 
irreducible.   Since $\tau$ is proper it is a submersion outside an analytic set $\tau^{-1}V'$,
where $V'\subset V$ has positive codimension, so that $\gamma=\tau_*\alpha$ is closed
and smooth on $V\setminus V'$. 
\end{remark}

Given  $\mu\in\GZ(X)$, for each  each of the irreducible
components  $\mu_j$ in \eqref{decomp}
consider the decomposition 
$\mu_j^{p_j}+\cdots+\mu_j^0$ as in \eqref{hund}.
We have the unique decomposition
\begin{equation}\label{deko}
\mu=\mu_{fix}+\mu_{mov},
\end{equation}
where  
\begin{equation}\label{dog5}
\mu_{fix}:=\sum_j \mu_j^{p_j},
\end{equation}
\begin{equation}\label{hund5}
\mu_{mov}=\sum_j \sum_{k<p_j} \mu_j^k,
\end{equation}
are called the  {\it fixed} and {\it moving} part of $\mu$, respectively. Notice that $\mu_{fix}$ is a cycle
in view of the dimension principle.
We say that each term in \eqref{dog5}
is a {\it fixed component} and each term in \eqref{hund5} a {\it moving component}
of $\mu$.
The reason for this terminology will be clarified in Section~\ref{vogelsec} but already
here we can present an illustrating example of a moving generalized cycle:

\begin{ex}\label{planmedel}
Assume that $X=\P^n_{[z_0:\ldots:z_n]}$ and let $\theta=dd^c\log\big (|z_1|^2+\cdots
+|z_n|^2\big )$.  
Then $\theta^{n-k}$, $k\ge 1$, 
is a generalized cycle in $\Pk^n$ of
dimension $k$ and with  Zariski support $\Pk^n$.  
To see this, let $\pi\colon
Bl_p\P^n\to\P^n$ be the blow-up at $p=[1\colon 0:\ldots :0]$
and notice that $\theta=\pi_*\hat \omega$, where
$\hat\omega$
is minus the first Chern form of the line bundle, with 
respect to the ``standard'' metric,  associated with the exceptional divisor $D$.
By repeated use of \eqref{ingen} we have that 
$\theta^{n-k}=\pi_* \hat \omega^{n-k}$ outside the origin. 
Since both sides are positive closed
currents it follows by the dimension principle that the equality must hold across $p$.
Thus $\theta^{n-k}$ is in $\GZ_k(\Pk^n)$ and by Lemma~\ref{gastub}
it is irreducible with Zariski support $\Pk^n$.  Thus it
has one single moving irreducible component. One can verify that  $\theta^{n-k}$ 
is indeed a mean value of all $k$-planes through $p$,  cf.\ \cite[Eq.\ (6.2)]{aswy}
with $f=(z_1, \ldots, z_n)$.  More conceptually, 
one can thus think of $\theta^{n-k}$ as such a $k$-plane moving around $p$.  
\end{ex}

\section{Equivalence classes of generalized cycles}\label{profet}
If   
$0\to S\to E\to Q\to 0$ is a short exact sequence
of Hermitian vector bundles over $X$ we say that 
\begin{equation}\label{syssling}
\hat c(E)-\hat c(S)\w\hat c(Q)
\end{equation}
is a {\it $B$-form} on $X$.
Let  $\beta$ be the component of bidegree $(k,k)$ of a $B$-form.  If $k=0$ then $\beta=0$
so let us assume that $k\ge 1$.
In view of \eqref{mor}
one can just as well use the differential-geometric definition of Chern form. From \cite[Proposition~4.2]{BC}
we get a smooth form $\gamma$ on $X$ of
bidegree $(k-1,k-1)$ such that $\beta=dd^c\gamma$.  In fact in  \cite{BC} only the case 
when $X$ is smooth is discussed. 
However, the construction of $\gamma$ is completely explicit and local, 
and locally we can extend our short exact sequence to a \nbh in a smooth ambient space 
and conclude that  $\gamma$ is smooth on $X$.

Notice for future reference that if
$\tau \colon W\to X$, then $\tau^*\beta$ is a component of a $B$-form if $\beta$
is.  
We say that  $\mu \in \GZ_k(X)$ is equivalent to $0$ in $X$,
$\mu\sim 0$,  if $\mu$ is a locally finite sum of currents of the form
\begin{equation}\label{kusin}
\rho=\tau_* (\beta \w \alpha) = dd^c \tau_*(\gamma\wedge \alpha),
\end{equation}
where $\tau\colon W\to X$ is proper, $W$ is smooth and connected, $\beta$ is
a component
of a $B$-form on $W$,  and
$\alpha$ is a product of components of Chern forms. 
If $\mu=\mu_0+\mu_1+\cdots$, where $\mu_k\in\GZ_k(X)$,  we say that
$\mu\sim 0$ if  $\mu_k\sim 0$ for each $k$.  
Let $\B(X)$ denote the $\Z$-module of generalized cycles on $X$ modulo this equivalence.
A class $\mu\in\B(X)$ has {\it pure
dimension $k$},  $\mu\in\B_k(X)$,  if $\mu$ has a representative  
 in $\GZ_k(X)$.   Thus $\B(X)=\oplus_k\B_k(X)$.

 If $E\to X$ is a Hermitian vector bundle, then for each $k$ we have the mapping
 \begin{equation}\label{skuta}
 \hat c_k(E)\w\colon \GZ_*(X)\to\GZ_{*-k}(X), \quad \mu\mapsto  \hat c_k(E)\w\mu.
 \end{equation}

\begin{prop}\label{pluttens} 
The mapping \eqref{skuta} induces  a mapping
\begin{equation}\label{skuta1}
c_k(E)\w\colon \B_*(X)\to\B_{*-k}(X)
\end{equation}
with the following properties: If  $F\to X$ is another vector bundle, then
\begin{equation}\label{skuta2}
c_\ell(F)\w c_k(E)\w \mu= c_k(E)\w c_\ell(F)\w \mu.
\end{equation}
If $f\colon W\to X$ is proper, then
\begin{equation}\label{skuta3}
f_*\big( c_k(f^* E) \w \mu\big)= c_k(E)\w f_*\mu
\end{equation}
for $\mu\in\B(W)$.
If $0\to S\to E\to Q\to 0$ is a short exact sequence on $X$,  then
\begin{equation}\label{skuta4}
c_k(E)\w\mu=\big(c(S)\w c(Q)\big)_k\w\mu.
\end{equation}
\end{prop}

\begin{proof}
First assume that $\hat\mu\in\GZ(X)$ and $\hat\mu\sim 0$. With the notation above we may assume that 
$\hat\mu=\tau_*(\beta\w\alpha)$, where $\tau\colon W\to X$ and 
$\tau$ is a $B$-form on $W$.  It follows that $\hat c_k(E)\w\hat\mu=
\tau_*(\beta\w\hat c_k(\tau^*E)\w\alpha)$ and hence by definition $\sim 0$. Thus 
$\hat c_k(E)\w$ is well-defined on $\B(X)$. We  must verify that it does not
depend on the particular choice of metric on $E$. 
To this end, assume that $0\to S\to E\to Q\to 0$ is a short exact sequence of Hermitian
vector bundles on $X$ and let $\beta$ be the component of bidegree $(k,k)$
of the associated $B$-form.  
Assume that $\tau\colon W\to X$ and $\hat\mu=\tau_*\alpha$
is an element in $\GZ(X)$. Then $0\to \tau^*S\to \tau^*E\to \tau^*Q\to 0$ is
a short exact sequence on $W$ and $\tau^*\beta$ is the component of bidegree $(k,k)$
of the associated $B$-form on $W$. It follows that
\begin{equation}\label{skuta12}
\beta\w\hat\mu=\tau_*(\tau^*\beta\w \alpha)\sim 0.
\end{equation}
If $S=0$ so that $E$  and $Q$ are isomorphic but with possibly different metrics, 
then $\beta=\hat c_k(E)-\hat c_k(Q)$ so we can conclude that 
$\hat c_k(E)\w\hat\mu-\hat c_k(Q)\w\hat\mu=0$ in $\B(X)$.  
Thus \eqref{skuta1} is well-defined. Now \eqref{skuta2} and \eqref{skuta3} are obvious and
\eqref{skuta4} follows from \eqref{skuta12}.
\end{proof}

\begin{remark}\label{bongo}
If $\beta$ is a component of \eqref{syssling}, but where all $\hat c$ are
replaced by $\hat s$, then still $\beta\w\alpha\sim 0$. In fact, if lower index $\ell$ denotes
component of bidegree $(\ell,\ell)$,  then
$$
\big(\hat s(E)-\hat s(S)\w\hat s(Q)\big)_k=
\sum_{\ell=0}^k \big( \hat c(E)-\hat c(S)\w\hat c(Q)\big)_\ell\w \big(\hat s(E)\w\hat s(S)\w\hat s(Q)\big)_{k-\ell},
$$
so the claim follows from Remark~\ref{kola}.  
It is clear that Proposition~\ref{pluttens} holds, with the same proof, if $c$ is replaced by $s$.
\end{remark}

Notice that if $h\colon X'\to X$ is a proper mapping and
$\mu\sim 0$, then
$h_*\mu\sim 0$ so we  have a natural mapping
$
h_* \colon \B(X)\to \B(X').
$

\begin{lma}\label{ska}
If $i\colon V\hookrightarrow X$ is a subvariety, then $i_*\colon \B(V)\to \B(X)$ is injective.
\end{lma}

\begin{proof}
Assume that $\mu\in\GZ(V)$ and $i_*\mu\sim 0$ in $\GZ(X)$. Then $i_*\mu=\sum\rho_j$, where
$\rho_j=(\tau_j)_*(\beta_j\w\alpha_j)$, $\tau_j:W_j\to X$, are as in 
\eqref{kusin}.
In view of Lemma~\ref{rest} we may assume that 
$\tau_j(W_j)\subset V$ for each $j$. 
For each  $j$ there is a map $\tau_j':W_j\to V$ such that $\tau_j=i\circ
\tau'_j$. 
Let $\rho_j'=(\tau'_j)_*(\beta_j\wedge\alpha_j)$.  
Then 
$$
i_*\mu=\sum_j \rho_j=i_*\sum_j\rho_j',
$$
so that 
$
\mu=\sum_j\rho_j'.
$
Thus
$\mu\sim 0$ on $V$.
 \end{proof}

\begin{prop}\label{svabb}
The mapping $\Zy(X)\to \B(X)$ is injective.
\end{prop}

Thus we can consider $\Zy(X)$ as a subgroup of $\B(X)$.

\begin{proof}
Assume that  $\mu=\sum_j a_j W_j\in \mathcal{Z}_k(X)$ and $\mu\sim 0$ in $\GZ_k(X)$. 
If $i\colon |\mu|\to X$ is the natural injection and 
$$
\mu'=\sum a_j\1_{W_j},
$$
then $\mu=i_*\mu'$. 
By Lemma~\ref{ska}, $\mu'\sim 0$ in $\GZ(|\mu|)$.
Since $\hat\mu$ has full dimension in $|\mu|$, and thus bidegree $(0,0)$, it must vanish
in view of \eqref{kusin}.
 \end{proof}

\begin{prop}\label{restriktion}
For each open subset $U$ of $X$ there is a 
natural restriction mapping
$r\colon \GZ(X)\to \GZ(U)$ that induces a mapping $r\colon \B(X)\to \B(U)$.
\end{prop}

\begin{proof}  Assume that $\mu\in\GZ(X)$ and $\mu=\tau_*\alpha$.
Then the restriction of the current $\mu$ to $U$ is equal to
$\tau'_*\alpha'$, where $\tau'$ and $\alpha'$ are the restrictions to
$U':=\tau^{-1}U$  of $\tau$ and $\alpha$, respectively.  Notice that
$\tau'\colon U'\to U$ is proper and that $\alpha'$ is a product of components of Chern forms 
since  $\alpha$ is.   
Since also the restriction to $U'$ of a $B$-form is a $B$-form, it follows
that $r$ is well-defined on $\B(X)$.
\end{proof}

\begin{lma} \label{brodd}
Assume that $\mu\in\GZ(X)$, $\mu\sim 0$, and that
\eqref{decomp} is its  decomposition in irreducible components.
Then $\mu_j\sim 0$ for each $j$.
\end{lma}

\begin{proof}
Using the notation from above, we can assume that $\mu$ is of the form 
$$
\mu=\sum_\ell (\tau_\ell)_*(\beta_\ell\w\alpha_\ell),
$$
where $\tau_\ell\colon W_\ell\to X$ are proper and $W_\ell$ are connected.  
It follows from the proof of Proposition~\ref{goltupp} that 
$$
\mu_j=\sum_{\tau_\ell(W_\ell)=|\mu_j|} (\tau_\ell)_*(\beta_\ell\w\alpha_\ell)
$$
and thus $\mu_j\sim 0$ by definition.
\end{proof}

Let $\hat\mu$ be a representative of $\mu\in\B(X)$ and let
$
\hat\mu=\sum_j \hat\mu_j
$
be its decomposition in irreducible components.  We claim that for each $j$
the corresponding class $\mu_j$ in $\B(X)$ is independent of the choice of $\hat\mu$.  
In fact,  assume that $\hat\nu$ is another representative with decomposition
$\sum_\ell \hat\nu_\ell$. The sums are (locally) finite and each term corresponds to
a unique irreducible  set, so by adding terms $0$ if necessary we have that
$$
\sum_j (\hat\mu_j-\hat\nu_j)\sim 0
$$
and hence by the lemma $\hat\mu_j-\hat\nu_j\sim 0$ for each $j$.  Now the claim follows, 
and taking into account
only the non-vanishing classes we get the unique decomposition
\begin{equation}\label{tyket}
\mu=\sum_j \mu_j,
\end{equation}
where $\mu_j$ are well-defined elements in $\B(X)$ with well-defined Zariski supports $|\mu_j|$.

In case this sum consists of just one non-zero term we thus have a
well-defined irreducible subvariety, and so the following definition is meaningful:

\begin{df} We say that $\mu\in\B(X)$  is {\it irreducible} if it has a representative $\hat\mu\in\GZ(X)$ that
is irreducible.  The  Zariski support $|\mu|$ of $\mu$ is then equal to  $|\hat\mu|$.  
\end{df}

We have the following simple consequences of the discussion above:

\begin{prop} 
(i)  If $\mu\in\B(X)$ is irreducible and $p=\dim|\mu|$, then we have a unique decomposition
$\mu=\mu^p+\cdots +\mu^0$, where $\mu^k\in\B_k(X)$.
 
\smallskip
\noindent (ii)   Any $\mu\in\B(X)$ has a unique decomposition
$\mu=\mu_1+\mu_2+\cdots$,
where $\mu_j\in\B(X)$ are irreducible.  
\end{prop}

\begin{df} 
In view of (ii) we define the Zariski support  $|\mu|$ 
as the union of the $|\mu_j|$.
\end{df}

From  Proposition~\ref{basic} and Lemma~\ref{ska} we get 

\begin{prop}\label{nixon}
  If $i\colon V\hookrightarrow X$, then the image of $i_*\colon\B(V)\to \B(X)$ is
precisely the $\mu$ in $\B(X)$ with Zariski support on $V$.
\end{prop}

That is,  we can identify the elements 
in $\B(V)$ with elements in $\B(X)$ with Zariski support contained in $V$.

Precisely as for generalized cycles we define $\mu_{fix}$ and $\mu_{mov}$ by \eqref{dog5} and \eqref{hund5},
respectively, and  get the unique decomposion, cf.\ \eqref{deko},
\begin{equation}\label{deko2}
\mu=\mu_{fix}+\mu_{mov},
\end{equation}
in $\B(X)$ in a fixed and a moving part, and in view of   Proposition~\ref{svabb}
the fixed part  is indeed a cycle in $X$.

\begin{remark}\label{hotto} 
Let $X$ be compact, $L\to X$ be a line bundle, and $\omega=c_1(L)$. The  {\it mass} 
$$
a:=\int_X \mu\w \omega^j
$$
of $\mu\in\GZ_j(X)$ is an integer that only depends on the class of $\mu$ in $\B_j(X)$ and of $L$. 
In fact,  
we may assume that  
$\mu=\tau_*\alpha$,
where $\alpha$ is a product of first Chern forms of line bundles
over $W$  and $\tau\colon W\to X$ is proper.
Then 
$\mu\w\omega^j=\tau_*(\alpha\w\tau^*\omega^j)$
and thus  
$$
a=\int_W \alpha\w \tau^*\hat\omega^j
$$
which is an integer since it is the  integral of a product of first Chern forms  of line bundles
and thus an intersection number. 
By \eqref{kusin} and Stokes' theorem it only depends on the class of $\mu$ and of $L$. 
When $j=0$ and $\dim|\mu|>0$ we think  of $\mu$ as $a$ points moving around on
$|\mu|$, cf.~Section~\ref{vogelsec}. 
\end{remark}

\section{The $\B$-Segre class}\label{segreclass}

Since any modification $\pi\colon X'\to X$ such that $\pi^*\J$ is principal
factorizes over the blow-up $Bl_\J X$ of $X$ along $\J$, it follows by Proposition~\ref{pluttens} and a standard
argument that $S(\J,X)$, as defined in the introduction, cf.~\eqref{segredef},  is a well-defined element in $\B(Z)$.
Recall the restriction map $r$ of Proposition~\ref{restriktion}. We claim that
\begin{equation}\label{musse}
S(\J|_U,U)= rS(\J,X).
\end{equation}
In fact, by linearity it is enough to check the case when $X$ is irreducible. 
If $\J$ is the $0$-ideal then \eqref{musse} is trivial. If not, 
let  
$\pi\colon X'\to X$ be a modification 
such that $\pi^*\J$ is principal. Then 
the restriction $\pi'\colon \pi^{-1}U\to U$ of $\pi$ is a modification where the pullback of
 $\J|_U$ is principal. Let $D'$ and $L'$ be the restrictions of
$D$ and $L$, respectively,  to $\pi^{-1}U$.  Then
$$
rS(\J,X)=r\pi_*\big([D]\w \frac{1}{1+c_1(L)}\big)= \pi'_*\big( [D']\w \frac{1}{1+c_1(L')}\big)=S(\J|_U,U).
$$

\begin{remark}\label{asegre}
In intersection theory, given a proper subscheme $W\to X$ there is a 
well-defined Chow class $s(W,X)$ in $\A(W)\simeq\A(Z)$, $Z=|W|$,  called the Segre class. As in the introduction
let us think of
$W$ as the nonreduced subspace of $X$ with structure sheaf $\Ok_W=\Ok_X/\J_W$,
where $\J_W$ is a coherent ideal sheaf over $X$.   
  Based on Chapter~4 in \cite{Fult} (the summary on
page 70 and Corollary~4.2.2) it follows that
$s(\J_W,X):=s(W,X)$ can be defined as $S(\J_W,X)$ 
in \eqref{segredef} if we  
interpret $c_1(L)^{j-1}\w[D]$  as the element 
$c_1(L)^{j-1}\cap[D]$ in the Chow group 
$\A(|D|)$ and $\pi_*$ as the push-forward of Chow classes, 
so that 
$s_k(\J_W,X):=(-1)^{k-1}\pi_*(c_1(L)^{k-1}\cap[D])$
is  an element in $\A(Z)$ for $k\ge 1$. 
Since $W$ is proper, $Z$ has positive codimension 
and therefore $s_0(\J_W,X)$ vanishes. 
\end{remark}

We shall now discuss concrete representatives of the $\B$-Segre class.  In particular, these representations
allow us to define the $\B$-Segre class not only on an analytic space but on a
generalized cycle $\mu$.   To this end we first 
consider Monge-Amp\`ere products on $\mu$, cf.\ \cite[Sections~5, 6]{aswy}.
Recall that $\sim$ is the equivalence relation defining $\mathcal{B}(X)$.

\begin{thm}\label{putig}
Assume that $\fff$ is a holomorphic section of a Hermitian bundle $E\to X$ and let $\J$
be the associated coherent sheaf with zero set $Z$.

\smallskip
\noindent (i)   For each $\mu\in\GZ(X)$ the limits
\begin{equation}\label{spott}
(dd^c\log|\fff|^2)^k\w\mu :=\lim_{\epsilon\to 0}
\big(dd^c\log(|\fff|^2+\epsilon)\big)^k\w\mu,
\quad k=0,1,2, \ldots,
\end{equation}
exist and are generalized cycles with Zariski support on $|\mu|$,
and the generalized cycles
\begin{equation}\label{spott2}
M^{\fff}_k\w \mu:=\1_Z(dd^c\log|\fff|^2)^k\w\mu, \quad k=0,1,2, \ldots,
\end{equation}
have Zariski support on $Z\cap|\mu|$.

\smallskip
\noindent (ii)
If $\mu\sim 0$, then $M^{\fff}_k\w\mu\sim 0$. 

\smallskip
\noindent (iii)
If $g$ is a 
holomorphic section of another vector bundle such that \footnote{Between norms $\sim$ has the standard meaning that 
there are constants $c,C>0$ such that $c|\fff|\leq |g|\leq C |\fff|$.} $|\fff|\sim |g|$, then
$M^{\fff}_k\w\mu\sim M^{g}_k\w\mu$.

\smallskip
\noindent (iv)
If $h\colon X'\to X$ is proper and $\mu'\in\GZ(X')$, then
\begin{equation*} 
M^{\fff}_k\w h_*\mu'=h_*\big( M^{h^*\fff}_k\w\mu'\big).
\end{equation*}
\end{thm}

The hypothesis in $(iii)$, which clearly holds if both $g$ and $\fff$ define $\J$,
 precisely means that the sheaves defined by $\fff$ and $g$ 
have the same integral closure, see, e.g., \cite{aswy}. We will refer to $(iv)$
as the {\it projection formula}.
We let 
$$
M^\fff\w\mu:=M^\fff_0\w\mu+M^\fff_1\w\mu+\cdots .
$$

\begin{proof}[Proof of Theorem~\ref{putig}]
We can assume that $\mu=\tau_*\alpha$, where $\tau\colon W\to X$ is proper
and $W$ is smooth and connected.  We first consider the case when $\tau^*\fff$ vanishes
identically on $W$, or equivalently, $|\mu|\subset Z$.   For $k\ge 1$ the limit in
\eqref{spott} trivially exists and is $0$, and so is \eqref{spott2}.  If $k=0$, then
\eqref{spott} is $\mu$ and $M^{\fff}\w\mu=\1_Z\mu=\mu$ as well.  Thus (i) holds, 
and (ii)-(iv) are easily verified. 

We can thus assume that $\tau^*\fff$ does not vanish identically on $W$ and hence it  defines
a subvariety of positive codimension.  Then $M^{\fff}_0\w\mu=\1_Z\mu=0$ since
$\mu$ is irreducible,  cf.\ Lemma~\ref{gastub}.  
Thus we may assume that
$k\ge 1$ 
and (possibly after a modification of $W$) that $\tau^*\J$ is principal on $W$.  This precisely means that 
$\tau^*\fff=\fff^0 \fff'$, where $\fff^0$ is a section of the line bundle
$L_D\to X'$ that defines the exceptional divisor $D$ and $\fff'$ is a
non-vanishing section of $\tau^*E\otimes L_D^{-1}=\Hom(L_D,\tau^*E)$. 
Thus $\fff'$ defines an isomorphism
between $L_D$ and a line subbundle of $\tau^*E$, and so  $L_D$ inherits a metric from 
$\tau^*E$ such that  $|\fff^0|_{}=|\tau^*\fff|_{}$.
If we let 
\begin{equation}\label{babian}
\hat\omega=-\hat c_1(L_D),
\end{equation}
we have
by the Poincar\'e-Lelong  formula that
\begin{equation}\label{doris}
dd^c\log|\tau^* \fff|^2=[D]+\hat\omega.
\end{equation}
By \eqref{ingen}, 
\begin{equation}\label{postman1}
(dd^c\log(|\fff|^2+\epsilon))^k\w\mu=
\tau_*\big((dd^c\log(|\tau^*\fff|^2+\epsilon))^k\w\alpha\big).
\end{equation}
By \cite[(4.6)]{A2},
\begin{equation}\label{postman2}
(dd^c\log(|\tau^*\fff|^2+\epsilon))^k\w\alpha\to
(dd^c\log|\tau^*\fff|)^k\w\alpha=
([D]+\hat\omega)\w \hat\omega^{k-1}\w\alpha, \quad \epsilon\to 0,
\end{equation}
where the middle expression is recursively defined by \eqref{recur}. The equality is
a simple consequence.    
We conclude that the limit \eqref{spott} exists for each $k\ge 1$ and that
\begin{equation}\label{rot}
(dd^c\log|\fff|^2)^k\w\mu=\tau_*\big(([D]+\hat\omega)\w \hat\omega^{k-1}\w\alpha\big).
\end{equation}
This is  a generalized cycles with Zariski support contained in
$\tau(W)=|\mu|$, cf.~\eqref{venus}.  
Since $|D|=\tau^{-1}Z$ we have by  \eqref{linalg}  that
\begin{equation}\label{trott}
M^{\fff}_k\w\mu=\1_Z (dd^c\log |\fff|^2)^k\wedge \mu=\tau_*\big([D]\w \hat\omega^{k-1}\w\alpha\big).
\end{equation}
Clearly it is in $\GZ(X)$ and has Zariski support contained in $Z\cap|\mu|$.  
Thus (i) is proved.

If $\alpha=\beta\w\alpha'$ for some component $\beta$ of a $B$-form,  then 
$$
M^{\fff}_k\w\mu=\tau_*\big([D]\w \hat\omega^{k-1}\w\beta\w\alpha'\big),
$$
and hence $\sim 0$. Thus (ii) holds.

If $g$ is a section as in (iii), then
we may assume that also $\tau^*g=g^0g'$. Since $|g|\sim |\fff|$
it follows that $g^0$ and $\fff^0$ define the same divisor and hence are sections of
the same line bundle. Hence their associated first Chern forms differ by a $B$-form 
on $W$.   
In view of \eqref{trott} and \eqref{babian}, 
$M^{\fff}_k\w\mu\sim M^{g}_k\w\mu$ and thus (iii) follows.
Finally, we get (iv) from \eqref{postman1}, with $h$ instead of $\tau$,  and \eqref{ingen}. 
\end{proof}

With the notation in the proof we have, cf.\ \eqref{trott}, 
\begin{equation}\label{korven}
M_k^{\tau^*\fff}\w\alpha=[D]\w \hat\omega^{k-1}\w\alpha.
\end{equation}
Moreover, cf.\ \eqref{orvar}, by definition
\begin{equation}\label{orvar2}
M^{\fff}=M^{\fff}\w\1_X.
\end{equation}

\begin{proof}[Proof of Theorem~\ref{thm2}]
We can assume that $X$ is irreducible.  If $\fff$ vanishes identically, then $M^\fff_0=1$ and $M^\fff_k=0$
for $k\ge 1$, so $M^\fff$ coincides with $S(\J,X)$ in this case.  Thus we may assume that $\J$ has positive codimension,
and that $\tau\colon W\to X$ is a modification such that $\tau^*\fff$ is principal.  It then follows from
\eqref{segredef}, \eqref{babian}, and \eqref{trott} with $\alpha=1$  that 
$M^{\fff}_k$ is a representative for $S_k(\J,X)$. 
Thus Theorem~\ref{thm2} follows. 
\end{proof}

\begin{ex} \label{packad}
If the proper map $\tau\colon W\to X$ is surjective and
  generically $m$-to-$1$, then $\tau_*\1_W=m\1_X$ and so
$
m M^{\fff}=\tau_* M^{\tau^*\fff}.
$
\end{ex}

\begin{remark}\label{tellus}
Assume that  $i\colon V\hookrightarrow X$ is a subvariety of pure codimension $p$.
By the projection formula, Theorem~\ref{putig} $(iv)$,
\begin{equation}\label{soting}
M_k^{\fff}\w [V]=i_*M^{i^*\fff}_k, \quad k=0,1,2,\ldots. 
\end{equation}
Notice that  the Segre class $S(i^*\J,V)$ on $V$ for $i^*\J\to V$ is represented by the generalized
cycle $M^{i^*\fff}$, cf.~Theorem~\ref{thm2}. With the identification given by Proposition~\ref{nixon} of elements in $\B(V)$ 
with elements in $\B(X)$ with Zariski support on $V$, thus the right hand side of 
 \eqref{soting} is a representative of $S(i^*\J,V)$. 
 {\it Warning:} The left hand side of \eqref{soting}  {\it is not} a product but an operator
acting on $[V]$.  In general one cannot recover $M^\fff$ from $i_*M^{i^*\fff}$, or
$S(\J,X)$ from  $S(i^*\J,V)$ even if $Z\subset V$. 
For instance, if $\J$ defines a regular embedding and $Z=V$,  then  
$i_*S(i^*\J,V)=[V]$ whereas $S(\J,X)=[V]\w s(N_\J X)$, see Proposition~\ref{gata1}. 
\end{remark} 

In view of \eqref{soting} the following definition is natural.  
 
\begin{df}
Assume that $\J\to X$ is defined by the section $\fff$ of the Hermitian
vector bundle $E\to X$. Given $\mu\in\B_p(X)$ and a representative
$\hat\mu\in\GZ_p(X)$, we define the $\B$-Segre class
$S_{k}(\J,\mu)$ as the class in $\B_{p-k}(Z\cap|\mu|)$ defined by $M_{k}^{\fff}\w\hat\mu$. 
We let $S(\J,\mu)=S_0(\J,\mu)+S_1(\J,\mu)+\cdots +S_p(\J,\mu)$. 
\end{df}

\begin{prop}\label{nylikhet}  
If $\alpha$ is a component of a Chern or Segre form,  then
\begin{equation}\label{samba1}
\1_Z (\alpha \w \mu)=\alpha \w \1_Z\mu, \quad \mu\in\GZ(X),
\end{equation}
and
\begin{equation}\label{samba2} 
M^{\fff}_k\w (\alpha\w\mu)=\alpha\w M_k^{\fff}\w\mu,\quad k=0,1,2, \ldots.
\end{equation}
\end{prop}

\begin{proof} 
Assume that $\mu=\tau_* a$ and $W$ is connected. Let $\xi=\tau^*\alpha$.
Now 
$\1_{\tau^{-1}Z}(\xi \w a)=\xi \w \1_{\tau^{-1}Z} a$ since both sides
vanish if $\tau^{-1}Z$ is a proper subvariety of $W$ and are equal to
$\xi\w a$ otherwise.  Thus \eqref{samba1} follows from
\eqref{ingen} and \eqref{linalg}, and 
 \eqref{samba2} follows from 
 \eqref{spott}, \eqref{spott2} and \eqref{samba1}.
 \end{proof}

Sometimes it is convenient with a limit procedure that directly gives 
$M_k^{\fff}\w\mu$ without first computing $(dd^c\log|\fff|^2)^k\w\mu$.

\begin{prop}\label{struts}
Let $\fff$ be a holomorphic section of a Hermitian bundle $E\to X$ and let
$$ 
M_{k,\epsilon}^{\fff}=\frac{\epsilon}{(|\fff|^2+\epsilon)^{k+1}}
(dd^c|\fff|^2)^k, \quad k=0,1,2,\ldots.
$$ %
If $\mu\in\GZ(X)$, then for $k\ge 0$, 
\begin{equation}\label{epsilon}
M_k^{\fff}\w\mu=\lim_{\epsilon\to 0} M_{k,\epsilon}^{\fff}\w\mu,
\quad k=0,1,2,\ldots.
\end{equation}
\end{prop}

Using a principalization, the proposition is reduced to 
the following lemma that can be verified
along the same lines as \cite[Proposition~4.4]{A2}, and we omit
the details.

\begin{lma}
Let $s$ be a section of a Hermitian line bundle $L\to W$ with
$\div s=D$ and let $\hat\omega=-\hat c_1(L)$. Then
$$
\frac{\epsilon}{(|s|^2+\epsilon)^{k+1}}(dd^c|s|^2)^{k}\to
[D]\w \hat\omega^{k-1}, \quad k\ge1.
$$
\end{lma}

\begin{remark}\label{geller}
One can define $M^{\fff}\w\mu$  
as the value at $\lambda=0$, via analytic continuation from $\Re\lambda\gg 0$,
of the expression
$$
M^{\fff,\lambda}\w\mu=\Big(1-|\fff|^{2\lambda}+\sum_{k\ge1}\dbar|\fff|^{2\lambda}\w\frac{\partial|\fff|^2}{2\pi i|\fff|^2}
\w(dd^c\log|\fff|^2)^{k-1}\Big)\w\mu,
$$
see \cite[Proposition~4.1]{A2} and \cite[Section~4]{aswy}.  
\end{remark}

\begin{remark}[Comparison to Green forms]
Recall that a $(p-1,p-1)$-current $g$ is a Green current of a closed subvariety $Z$
of codimension $p$ of a complex manifold $X$ 
if $dd^c g+[Z]=\omega$, where $\omega$ is a smooth form.  If $g$ is smooth outside $Z$ it is called a Green form. 
The calculus of Green forms, based on the $*$-product, is an important tool
in the study of height in arithmetic intersection theory, see, e.g., \cite{BGS,GS}. 
In particular, Fulton's intersection theory is recovered in the proper intersection case.

In the case $p=1$, if $s$ is a section of a Hermitian line bundle that defines $Z$, then $g=-\log|s|^2$ is a Green form
in virtue of the Poincar\'e-Lelong formula \eqref{pl}. 
In fact, these are the only Green forms in the case $p=1$.  
The existence of Green forms of so-called logarithmic type for $p>1$ is a more
delicate matter, see \cite{BGS}. That $g$ is of logarithmic type means that
$g=\tau_*g'$ under a proper mapping $\tau\colon W\to X$ such that locally in $W$ there are coordinates
$z$ such that $g'=\sum_k a_k \log|z_k|^2 + a_0$, where
$a_j$ are smooth and closed.  This can be compared to our definition of generalized cycles.

If $Z$ is defined by the section $\sigma$ of the Hermitian vector bundle $E\to X$,
and $\gamma:=-\log|\sigma|^2(dd^c\log|\sigma|^2)^{p-1}$, then, cf.\ \eqref{recur} and Corollary~\ref{snabela}, 
$$
dd^c \gamma= -(dd^c|\sigma|^2)^{p}=-\1_Z(dd^c|\sigma|^2)^{p}-\1_{X\setminus Z}(dd^c|\sigma|^2)^{p}=
-[Z]-\1_{X\setminus Z}(dd^c|\sigma|^2)^{p},
$$
so that $\gamma$ is kind of a Green form. Unless $p=1$, however, $-\1_{X\setminus Z}(dd^c|\sigma|^2)^{p}$ 
is not smooth but only
the push-forward under a modification of a smooth form, cf.\ \eqref{rot}.  
\end{remark}

\section{Multiplicities of a generalized cycle}
In view of \eqref{imam} it is natural to define
the multiplicity of $\mu\in\GZ_k(X)$ at $x\in X$ as the Lelong
number at $x$. However, $\mu$ is not necessarily positive so it is not immediately 
clear that the Lelong number exists.  Here is our
formal definition:
Let $\fff_x$ be a section of a Hermitian vector bundle 
in an open \nbh $U$ of $x$ such that 
$\fff_x$ generates  the maximal ideal $m_x$ at $x$. Since $M^{\fff_x}\w\mu$ has support
at $x$ it follows from the dimension principle,
Proposition~\ref{kvick},  that $M^{\fff_x}\w\mu=M^{\fff_x}_k\w\mu$. 
Moreover, in view of the proof of this proposition,
$M^{\fff_x}_k\w\mu=a[x]$ for some real number $a$. By Theorem~\ref{putig}~(iii), the number $a$ 
is independent of the choice of ${\fff_x}$. Part (ii) of the same theorem implies that $a$ only depends
on the class of $\mu$ in $\B(U)$.
By an argument as in the beginning of 
Section~\ref{segreclass} we see that it is also independent of the choice of neighborhood $U$ of $x$.
Altogether the definition
\begin{equation}\label{plutt}
\mult_x \mu=\int M^{{\fff_x}}\w\mu
\end{equation}
is meaningful.  
If $U$ is small enough we can assume that $E$ is trivial, with a trivial 
metric, and then $\mult_x\mu$ coincides with the Lelong number $\ell_x\mu$ 
if $\mu$ is  positive,  see, e.g., \cite[Lemma~2.1]{aswy}
and Remark~\ref{geller}.

\begin{prop}\label{gatsten} 
The multiplicity of $\mu\in\GZ_k(X)$ at $x$ is an integer and it 
only depends on its class in $\B_k(X)$.
\end{prop}

\begin{proof}  
Let  $\mu=\tau_*\alpha$, where $\tau\colon W\to X$ is proper
and $W$ is connected. First assume that $\tau(W)=\{x\}$.
 Thus  $M^{\fff_x}\w\mu=\mu$.   Since $\tau$ is proper,
$W$ is compact, so by \eqref{plutt},  
$$
\mult_x\mu=\int\mu=\int_W\alpha
$$
which is an intersection number, cf.~Remark~\ref{hotto}, and hence an integer.  
Next we assume that $x\in \tau(W)$ and that $\tau(W)$ has positive dimension.  
As in the proof of
Theorem~\ref{putig},  with $\fff={\fff_x}$ and $X=U$, cf.~\eqref{trott} and \eqref{plutt},
we can assume that 
$$
M^{\fff_x}\w\mu=\sum_{k\ge 1}\tau_*\big([D]\w\hat\omega^{k-1}\w\alpha\big).
$$
Only the term with $k=\dim\mu$ can give a contribution and
$$
\mult_x \mu=\int M^{\fff_x}_k\w\mu=\int \tau_*\big([D]\w\hat\omega^{k-1}\w\alpha\big).
$$
Writing $D=a_1 D_1 + a_2 D_2 +\cdots $, where $D_j$ are irreducible and compact,
we therefore have that  
$$
\mult_x\mu= a_1 \int_{D_1}\hat\omega^{k-1}\w \alpha +
\cdots
$$
and hence an integer, since each integral is an intersection number. 
\end{proof}

Assume that $\mu$ is irreducible. If it has dimension $0$ and is moving, i.e., $\dim |\mu|>0$, then 
$\mult_x\mu=0$ at each point. In fact,  
$(\mult_x\mu) [x]=\1_x\mu=0$ by the definition of irreducibility.
However, as is illustrated by 
Example~\ref{planmedel}, if $\mu$ has positive dimension, $\mult_x\mu$ can be nonzero at
certain points even if $\mu$ is moving.

\begin{proof}[Proof of Theorem~\ref{genking}]
It is well-known that the blow-up $\pi\colon Bl_\J X\to X$ of $X$ along $\J$ only depends on the 
integral closure class of $\J$.  Since $S(\J,X)$ is defined just in terms of the blow-up,
cf.~\eqref{segredef}, it 
only depends on the integral closure class of $\J$.  

By definition the distinguished varieties are precisely the
sets $\pi(D_j)$ where $D_j$ are the irreducible components of the
exceptional divisor $D$ of the blow-up, see, e.g., \cite{Laz} or \cite{aswy}. 

The remaining statements of Theorem~\ref{genking} are purely local and can be verified in the
following way: 
Fix a point $x \in X$. In a suitable \nbh $U$ 
of $x$ there is a section $\fff$ of a trivial vector bundle $E\to U$
 that generates $\J$ there.   
By Proposition~\ref{restriktion},
$\mult_x S_k(\J,X)=\mult_x S_k(\J|_U, U)$. If we choose  a trivial
metric, then $M^{\fff}_k$ coincides with $M^\fff_k$ defined in
\cite{aswy}, and from \cite[Theorem~1.1]{aswy} we have that
$\mult_xM^{\fff}_k=e_k(\J,x)$.  
Because of the uniqueness of the decomposition \eqref{deko}
in fixed and moving components applied to
$S(\J,X)$ all the statements now follows from \cite[Theorem~1.1]{aswy}. 
 \end{proof}

We have the following consequence of Proposition~\ref{nylikhet}.

\begin{lma}\label{spe}
If $\mu\in\B_k(X)$ and $\gamma$ is a component of a Chern or Segre form of positive bidegree,  then
$
\mult_x(\gamma\w\mu)=0
$
for each  $x$.
\end{lma}

\begin{proof} 
Let ${\fff_x}$ generate the maximal ideal at $x$, and write $\gamma=dd^c g$ in a \nbh of $x$. 
By Proposition~\ref{nylikhet} and Stokes' theorem, noting that $M^{\fff_x}\w\mu$ has support at $x$, 
$$
\mult_x(\gamma\w\mu)=\int M^{\fff_x}\w(\gamma\w\mu)
=\int \gamma\w M^{\fff_x}\w\mu=\int dd^c(g\w M^{\fff_x}\w\mu)=0.
$$
\end{proof}

\begin{ex}[Example~\ref{planmedel} continued]
It follows from Lemma~\ref{spe}  that $\mult_x\theta^{n-k}=0$ for $x\neq p$.
From the geometric interpretation as a mean value of $k$-planes through $p$, 
or by a direct computation of $M^{\fff_x}  \w \theta^{n-k}$,
one can verify that $\mult_p\theta^{n-k}=1$.
\end{ex}

In view of Theorem~\ref{genking}, $\mult_x (M_j^{\fff}\wedge \1_X)=\mult_x M_j^{\fff}=e_j(\J,x)$.
For a general $\mu$ in $\GZ_k(X)$ or $\B_k(X)$  we can define the Segre numbers 
$e_j(\J,\mu,x):=\mult_x (M^\fff_j\w\mu)$.

\section{Regular embeddings}\label{regular}
Assume that $\J\to X$ defines a regular embedding 
$i\colon Z_\J\to X$ of codimension $\kappa$, cf.~the introduction and Remark~\ref{asegre}.
As before $Z$ denotes the associated reduced space,
i.e., the zero set of $\J$. 
It is well-known that $\J/\J^2$ is locally free and thus is the sheaf of sections of a vector bundle
known as the conormal bundle of $Z_{\J}$ in $X$. We will denote its dual by $N_{\J}X$, refer to it 
as the normal bundle of   
$Z_{\J}$ in $X$, and view it as a
holomorphic vector bundle over the reduced space $Z$. We will use the following alternative ad hoc definition of 
$N_{\J}X$ and its sections: A section $\xi$ of $N_\J X\to Z$
is a choice of holomorphic $\kappa$-tuple $\xi(s)$ locally on $Z$ for each
local minimal set $s=(s_1,\ldots,s_\kappa)$ of generators for $\J$ so that
\begin{equation}\label{sss}
g\xi(s)=\xi(gs)\ \text{on} \  Z
\end{equation}
for any locally defined holomorphic matrix $g$ that is invertible in a neighborhood 
of $Z$. 
This defines a vector bundle over $Z$ since for any two such choices $s,s'$ there is an invertible
matrix $g$ such that $s'=gs$ on the overlap in a \nbh of $Z$.
The connection between $N_{\J}X$ and $\J/\J^2$ is the non-degenerate pairing
$(\xi(s), h+\J^2)\mapsto \xi(s)\cdot h_s$, where $h_s$ is a tuple, unique mod $\J^2$,
such that $h=s\cdot h_s$.
\begin{ex}
If $Z$ is smooth and $Z_\J$ reduced, then for any $s$ as above,
the $ds_j$ are linearly independent on 
$Z$ and vanish on $TZ$. Notice that $gds=d(gs)$ on $Z$.
Therefore, $v\mapsto \xi(s):=(ds_1\cdot v,\ldots,ds_{\kappa}\cdot v)$ defines
an injective mapping, hence an isomorphism, $TX|_Z/TZ \to N_\J X$.
In this case therefore $N_\J X$ is the usual normal bundle of complex differential geometry.
\end{ex}

\begin{remark}
Several results of this section are well-known, at least in the algebraic context. 
For completeness and reference we give analytic proofs.
\end{remark}

\begin{lma}\label{polly}
Assume that $F\to X$ is a vector bundle with a holomorphic section $\varphi$ that
defines $\J$. Then there is a canonical embedding
\begin{equation}\label{apa}
i_{\varphi} \colon N_\J X\to F|_Z.
\end{equation}
\end{lma}

\begin{proof}
For each minimal set of generators $s$ of $\J$ in some open connected $U\subset X$
there is a unique $\tilde{A}(s)$ in 
$\Hom(U\times\C^\kappa, F|_U)$ such that $\tilde{A}(s) s=\varphi$.  
Set $A(s)=\tilde{A}(s)|_Z\in \Hom(Z\times\C^\kappa, F|_Z)$. Since $\varphi$
generates $\J$ it follows that $A(s)$ is pointwise injective.
Since 
$
\varphi=\tilde{A}(gs)gs
$
it follows that 
$
A(gs) g= A(s).
$
If $\xi$ is a section of $N_{\J}X$  therefore
\begin{equation}\label{kobolt}
A(gs)\xi(gs)=A(gs) g\xi(s)=A(s)\xi(s).
\end{equation}
Thus we can define $i_\varphi$ as
\begin{equation}\label{lar1}
\xi\mapsto i_\varphi \xi, \quad \xi(s)\mapsto A(s)\xi(s).
\end{equation}
Since $A(s)$ is pointwise injective it follows that 
$i_\varphi$  is injective.
\end{proof}

In particular, if $\text{rank}\, F=\kappa$, then we have an isomorphism 
\begin{equation}\label{polly2}
i_{\varphi}\colon N_\J X\simeq F|_Z.
\end{equation}

\begin{proof}[Proof of Proposition~\ref{gata2}]
We let $q\colon \P(F)\to X$ be the projectivization of $F$ and let
 $\Ok_F(-1)\to\P(F)$ be the tautological line bundle sitting in $q^*F$,
 equipped with the Hermitian metric inherited from $F$. 
 The line bundle $\Ok_{N_\J X}(-1)\to\Pk(N_\J X)$ is defined in the same way, with the Hermitian metric inherited
 from the normal bundle $N_\J X\to Z$, which in turn has the metric induced by \eqref{apa}.
Moreover,
we let   $p\colon Bl_{\J} X\to X$ be the blow-up of $X$ along $\J$ and let $L_D\to Bl_{\J} X$
be the line bundle associated with the exceptional divisor.  There are injective holomorphic mappings
$j,\tilde j, \psi, \tilde\psi$ such that the diagram 
\begin{equation}\label{diagram2}
\xymatrix{
\mathcal{O}_{N_\J X}(-1) \ar@{^{(}->}[r]^{\tilde j} \ar[d] & L_D  \ar@{^{(}->}[r]^{\tilde\psi} \ar[d] & \mathcal{O}_F(-1) \ar[d]\\
\mathbb{P}(N_{\J} X) \ar@{^{(}->}[r]^{j} \ar[d]^{\pi} & Bl_{\J}X \ar@{^{(}->}[r]^{\psi} \ar[d]^p & \mathbb{P}(F) \ar[d]^q\\
Z \ar@{^{(}->}[r]^i & X \ar[r]^{=} & X
}
\end{equation} 
commutes and such that furthermore the Hermitian metric on  $\Ok_{N_\J X}(-1)$ coincides with the metric
it inherits from $\Ok_F(-1)$ via the first row.

Let us first explain the mapping $j$.  Given a minimal set of local generators $s$ of $\J$ as above
in say an open set $U\subset X$ we can represent $Bl_\J U\to U$ as
\begin{equation} 
Bl_\J U=\{(x,[t])\in U\times \P^{\kappa-1}; \  t_i s_j(x)-t_js_i(x)=0\}.
\end{equation}
If we choose $s'=gs$ in $U'$,  then we have a similar representation but with $[t']=[gt]$ on the overlap $U\cap U'$.
Recall that at $x\in Z$ the fibre of $N_\J X$ consists of all $\xi(s)\in\C^\kappa$ such that $\xi(s')=g\xi(s)$, 
cf.~\eqref{sss}.  We thus have the natural injection 
$$
j\colon \P(N_\J X)\to Bl_\J X, \quad  (x,[\xi(s)])\mapsto (x,[t]).
$$

In $Bl_\J X\setminus p^{-1}Z$ we define $\psi$ by $\psi(p^{-1}x)=(x,[\varphi(x)])$.  
If  $x\in U\setminus p^{-1}Z$, then $p^{-1}x=(x,[s(x)])$ and $\tilde A(s)s=\varphi$,
cf.\ the proof of Lemma~\ref{polly}, so 
we have that  
\begin{equation}\label{orm}
\psi(x,[t])=(x,[\tilde A(s) t])
\end{equation}   
since $[s(x)]=[t]$. 
Since $A=\tilde A|_Z$ is injective, \eqref{orm} provides an injective extension of $\psi$
across $p^{-1}Z$ in $U$.   This extension is well-defined on overlaps because
if $s'=gs$, then $[t']=[gt]$ and by \eqref{kobolt} hence $[A(s')t']=[A(s)t]$. 
For $x\in Z$ thus $(x,[\xi(s)])$ in $\P(N_\J X)$ is mapped to $(x,[t])$   and by  $\psi$ in turn to 
$(x, [A(s)\xi(s)])$ so the composed mapping 
$\psi\circ j$ is equal to the mapping $\P(N_\J X)\to\P(F)$
induced by the canonical embedding $i_\varphi$, cf.~\eqref{lar1}. 
Thus the lower ``half'' of the diagram is defined and commutes.

We now define the mapping $\tilde\psi$.
Since $p^*\J$ is principal we recall from the proof of Theorem~\ref{putig} (with $\fff=\varphi$
and $p^*\varphi=\varphi^0\varphi'$)
that   $L_D\to Bl_\J X$ can be identified with a line subbundle
of $p^*E\to Bl_\J X$ via the mapping $\varphi'$.  
Since by commutativity $p^*F$ is the restriction of $q^*F$ to $\psi(Bl_\J X)$ we have an injective mapping
$L_D\hookrightarrow q^*F$.  We must verify that it actually takes values in $\Ok_F(-1)\subset q^*F$.  
By continuity it is enough to check that this holds over $Bl_\J X\setminus p^{-1}Z$.
However, there the section $\varphi^0$ is non-vanishing and mapped onto
$\varphi^0\varphi'=p^*\varphi$.  Thus $(p^{-1}x,\varphi^0(x))$ is mapped onto
$(x,[\varphi(x)], \varphi(x))$ which is in $\Ok_F(-1)$ by definition.  

It remains to explain $\tilde j$. Notice that $i_\varphi$ induces an embedding
$\pi^*N_\J X\hookrightarrow q^*F$ and hence also an embedding 
$\tilde i_\varphi \colon \Ok_{N_\J X}(-1)\hookrightarrow \Ok_F(-1)$.  Since $\tilde\psi$   is already defined,
there is a unique mapping $\tilde j$ so that $\tilde\psi \circ \  \tilde j=\tilde i_\varphi$ and the diagram commutes.
If $\xi$ is a vector in $\Ok_{N_\J X}(-1)$, then by definition $|\xi|=|\xi|_{N_\J X}$ equals $|i_\varphi\xi|_F$. However,
the norm of $\xi$ induced by the top line is  $|\tilde i_\varphi\xi|$  which in turn is $|i_\varphi\xi|_F$ as well.
Thus the claims about \eqref{diagram2} are proved.

\smallskip
As before, cf.~\eqref{babian}, we let $-\hat\omega=\hat c_1(L_D)$.
By \eqref{diagram2}, $-j^*\hat{\omega}$ is the 
first Chern form of $\mathcal{O}_{N_\J X}(-1)\to\mathbb{P}(N_{\J} X)$
and so, by definition, cf.~\eqref{orgie}, 
$$
\hat{s}(N_{\J}X)=\pi_*\big(\frac{1}{1-j^*\hat\omega}\big)   
$$
Each irreducible component $Z_\ell$ of $Z$ corresponds to an irreducible component $D_\ell=p^{-1}Z_\ell$
of $j(\mathbb{P}(N_\J X))=|D|$
and $[D]=\sum_\ell a_\ell [D_\ell]$
for some integers $a_\ell$.  
Since $\hat{s}(N_{\J}X)$ and $\hat\omega$ are smooth it follows that
\begin{equation}\label{orvar1}
\1_{Z_\ell} \hat s(N_\J X)=\pi_*\1_{j^{-1}D_\ell}\big(\frac{1}{1-j^*\hat\omega}\big).
\end{equation}
Multiplying by $a_{\ell}$ and applying $i_*$ to the left hand side of \eqref{orvar1} we get 
$$
\hat s(N_\J X) \w a_\ell [Z_\ell].
$$
The same action on the right hand side of \eqref{orvar1} gives, using that $i_*\pi_*=p_*j_*$,
$$
p_*\big( a_\ell[D_\ell]\w \frac{1}{1-\hat\omega}\big).
$$
Summing up we get
\begin{equation}\label{stopp}
\hat s(N_\J X) \w \sum_\ell  a_\ell[Z_\ell]= p_* \big([D]\w \frac{1}{1-\hat\omega}\big)
=p_*\big(\sum_k [D]\wedge\hat{\omega}^k\big)=M^{\varphi},
\end{equation}
where the last equality follows from \eqref{trott} (with $W=Bl_\J X$ and $\alpha=1$); notice that 
$M^{\varphi}_0=0$ here.  It remains to see that
$\sum_{\ell}a_{\ell}[Z_{\ell}]$ is the fundamental cycle $[Z_{\J}]$:
Since $\hat s_0(N_\J X) =1$ we get from \eqref{stopp} that 
$
M_\kappa^{\varphi}=\sum_\ell  a_\ell[Z_\ell].
$
The same argument applied to the section $s=(s_1,\ldots,s_{\kappa})$ of the trivial 
$\text{rank}\,\kappa$-bundle (with trivial metric) over $U$ gives that $M_{\kappa}^{s} =\sum_\ell  a_\ell[Z_\ell]$.
It follows from \cite[Ch.~3.16, Thm~3]{Ch} that $M^s=M^s_\kappa$ is the proper intersection 
$[\div s_1]\w\cdots\w[\div s_\kappa]$,  and it follows from \cite[Ch.\ 7]{Fult} that this product is
the fundamental cycle $[Z_\J]$ in case of a regular embedding. 
\end{proof}

\begin{remark}\label{kossa}
In the proof above we did not describe $\tilde j$ explicitly. 
With the notation above,  in a set $p^{-1}U$ we can consider $L_D\to Bl_\J U$ as the line subbundle of
$Bl_\J U\times\C^\kappa$ such that the fibre over a point $(x,[t])$ is the line $\{\lambda t\in\C^\kappa;\  \lambda\in \C\}$.  
Thus $\tilde j$ maps the point $(x, [\xi(s)], \xi(s)]$ to $(x,[t],t)$ in $L_D$.
\end{remark}

It is well-known, and indeed follows from the proof above, that $Bl_{\J}X$ can be seen as the 
closure in $\mathbb{P}(E)$ of the graph
$\{(x,[\varphi(x)])\in\mathbb{P}(E);\, x\in X\setminus Z\}$; then the mapping $\psi$ is 
of course just the natural inclusion.

\begin{cor}\label{aeswytillagg}
Let  $i\colon V\hookrightarrow X$ be an irreducible subvariety and assume 
that $i^*\varphi$ defines a regular embedding
of codimension $\kappa$ in $V$. Then
$$
M^{\varphi}\wedge [V] = \hat{s}(N_{\J}X)\wedge [Z_{\J}]\wedge [V].
$$
\end{cor}

\begin{proof}
From Proposition~\ref{gata2} we have for degree reasons  that
$M^\varphi_\kappa=[Z_\J]$.
Since $Z$ and $V$ intersect properly by assumption, $[Z_{\J}]\wedge [V]$ makes sense
and, moreover, 
$
M^\varphi_\kappa\w[V]=[Z_\J]\w[V],
$
cf.~\cite[Section~2.4]{aswy} and \eqref{epsilon}. 
On the other hand, from Proposition~\ref{gata2} applied to $i^*\J$, 
cf.~\eqref{soting}, 
$
M^\varphi_\kappa\w[V]=i_* M^{i^*\varphi}_\kappa= i_* [Z_{i^*\J}].
$
Thus
\begin{equation}\label{fabel}
i_* [Z_{i^*\J}]=[Z_\J]\w[V].
\end{equation}
If the tuple $s=(s_1,\ldots,s_{\kappa})$ generates $\J$, then $i^*s$ generates $i^*\J$
and the transition matrices of $N_{i^*\J}V$ are the restriction to $V$ of the transition
matrices of $N_{\J}X$.  Thus
$%
N_{i^*\J}V=i^*N_{\J}X.
$      
Moreover, the Hermitian metric on $N_{i^*\J}V$ is inherited from $N_{\J}X$ so that
\begin{equation}\label{grobian}
\hat s\big(N_{i^*\J}V\big)=i^*\hat s\big(N_{\J}X\big).
\end{equation}
By Proposition~\ref{gata2},  \eqref{soting}, \eqref{fabel} and \eqref{grobian} we thus get
\begin{eqnarray*}
M^{\varphi}\wedge [V] &=& i_*M^{i^*\varphi} = i_*\big(\hat{s}(N_{i^*\J}V)\wedge [Z_{i^*\J}]\big)
=i_*\big(i^*\hat{s}(N_{\J}X)\wedge [Z_{i^*\J}]\big) \\
&=& \hat{s}(N_{\J}X)\wedge [Z_{\J}]\wedge [V].
\end{eqnarray*}
\end{proof}

\begin{proof}[Proof of Proposition~\ref{gata1}]
Notice that the right-hand side of the equation in the formulation of the proposition is well-defined in view
of Proposition~\ref{pluttens}.
In view of Theorem~\ref{thm2} and \eqref{skuta1} the proposition follows immediately from Proposition~\ref{gata2}
if there is  a vector bundle with a section defining $\J$. If not we still have, cf.~Remark~\ref{kossa},
 the commutative diagram
\begin{equation}\label{diagram}
\xymatrix{
\mathcal{O}_{N_\J X}(-1) \ar@{^{(}->}[r]^{\tilde j} \ar[d] & L_D  \ar[d]  \\ 
\P(N_\J X)    \ar@{^{(}->}[r]^{j} \ar[d]^{\pi} & Bl_\J X  \ar[d]^p\\
Z \ar@{^{(}->}[r]^i & X.
}
\end{equation}  
By definition, cf.~\eqref{segredef}, recalling that $\kappa\ge 1$,
$S(\J,X)=p_*\big([D]\wedge 1/(1+c_1(L_D))\big)$ and, by \eqref{orgie}, $s(N_{\J}X)=\pi_*\big(1/(1+j^*c_1(L_D))\big)$.
Thus the result follows as in the proof of Proposition~\ref{gata2},  replacing computations
in $\GZ(X)$ by analogous ones in $\B(X)$. 
\end{proof}

\begin{prop}\label{vaarraekning}
Let $\fff$ be a holomorphic section of a Hermitian bundle $E\to X$ defining 
the regular embedding $\J$ and 
let $\varphi$ be a holomorphic section of a Hermitian bundle $F\to X$ defining 
a regular embedding of codimension $1$. Suppose that the section $\fff+\varphi$ of the Hermitian bundle
$E\oplus F\to X$ defines a regular embedding of codimension $\kappa+1$.
Then 
$$
M^{\fff+\varphi}=M^{\fff}\wedge M^{\varphi}=M^{\varphi}\wedge M^{\fff}.
$$
\end{prop}

\begin{proof}
Let us first assume that $\kappa=1$. Then the statement is symmetric in $\fff$ and $\varphi$;
$\fff$ and $\varphi$ are sections of line subbundles $L_1\subset E$ and $L_2\subset F$
defining divisors $D_1$ and $D_2$, respectively. 
By Proposition~\ref{gata2},  $M^{\varphi}=\hat{s}(L_2)\wedge [D_2]$,  
cf.~\eqref{polly2}, and so, by Corollary~\ref{aeswytillagg}, 
since $\fff|_{|D_2|}$ is generically non-vanishing, 
$$
M^{\fff}\wedge M^{\varphi} = \hat{s}(L_1)\wedge \hat{s}(L_2)\wedge [D_1]\wedge[D_2].
$$
We have that $\fff+\varphi$ is a section of the Hermitian bundle $\mathcal{E}:=L_1\oplus L_2 \subset E\oplus F$ defining a regular embedding 
of codimension $2$. Denote the corresponding ideal by $\J'$, its zero set by $Z'$, and notice that $N_{\J'}X=\mathcal{E}|_{Z'}$. 
By Proposition~\ref{gata2} we have
$$
M^{\fff+\varphi}=\hat{s}(\mathcal{E})\wedge [Z_{\J'}]=
\hat{s}(\mathcal{E})\wedge [D_1]\wedge [D_2],
$$
where the last equality follows as in the end of the proof of Proposition~\ref{gata2}. 
It follows from \eqref{smal} and \eqref{mor} that 
$\hat{s}(\mathcal{E})=\hat{s}(L_1)\wedge \hat{s}(L_2)$. This concludes the proof when $\kappa=1$.

Now assume that  $\kappa\ge 2$.  
Let $p\colon Bl_\J X\to X$ be the blow-up of $X$ along $\J$. Then both $p^*\fff$ and $p^*\varphi$ define
principal ideals and it is readily verified that $p^*\fff + p^*\varphi$ defines a regular embedding in
$Bl_\J X$ of codimension $2$.  Since $p$ is a modification it is generically an isomorphism and
hence  from  Example~\ref{packad},
Theorem~\ref{putig}~(iv), and the  case $\kappa=1$ proved above,  we get 
$$
M^{\fff+\varphi}=p_* M^{p^*\fff+p^*\varphi}=p_*\big(M^{p^*\fff}\w M^{p^*\varphi}\big)=
M^\fff\w p_*\big(M^{p^*\varphi}\big)=M^\fff\w M^\varphi.
$$
\end{proof}

\begin{remark} It is not necessary to assume that $\fff$ defines a regular embedding;
the proof only relies on the fact that $p^*\fff + p^*\varphi$ 
defines a regular embedding.  One can therefore formulate a variant of Proposition~\ref{vaarraekning} that
is a global version of  Lemma~9.2 in \cite{aswy}.
\end{remark}
 
\begin{ex}
Let $\tau\colon X'\to X$ be a section of
a locally trivial fibration $\pi\colon X\to X'$ with one-dimensional fibers,
let $\varphi$ be a section of a Hermitian line bundle $L\to X$ defining $\tau(X')$, and let $\fff'$ be a section of a 
Hermitian bundle $E'\to X'$ defining a regular embedding. If $\fff=\pi^*\fff'$, then
$$
\tau_* M^{\fff'}=\hat c(L)\wedge M^{\fff+\varphi}.
$$
To see this, notice first that it follows from Proposition~\ref{gata2} and \eqref{smal} that 
$\hat{c}(L)\wedge M^{\varphi}=[\varphi=0]=[\tau(X')]$. Thus by Proposition~\ref{vaarraekning},
$\hat c(L)\wedge M^{\fff+\varphi}=M^{\pi^*\fff'}\wedge \hat{c}(L)\wedge M^{\varphi}=M^{\pi^*\fff'}\wedge [\tau(X')]=\tau_*M^{\fff'}$.
\end{ex}

Let $i\colon Z_\J\hookrightarrow X$ be a regular embedding of codimension $\kappa\ge 1$.
We conclude with a short discussion of the $\B$-Gysin mapping \eqref{bgysin}. 
It is further studied in \cite{AESWY2}.
In analogy with
Chow theory, cf.\ \cite[Ch.\ 6]{Fult}, one can think
of $(c(N_\J X)\w S(\J,X)\w \mu)_{k-\kappa}$ as an intersection
of $Z_\J$ and $\mu$ in $\B(X)$.  We assume that $\J$ is defined 
by the section $\varphi$ of the Hermitian bundle $F\to X$ so we can also
consider the more explicit mapping \eqref{jupiter}.

First, let  $\gamma\in\GZ_k(X)$ be a product of components of 
Chern or Segre forms. We claim that 
\begin{equation}\label{putin}
(\hat c(N_\J X)\w M^{\varphi}\wedge\gamma)_{k-\kappa}=[Z_\J]\w\gamma=i_* i^*\gamma, 
\end{equation}
so that \eqref{jupiter} 
can be seen as a generalization to $\GZ(X)$  of $i_* i^*$.
In fact, by \eqref{samba2}, $M^\varphi\w\gamma=\gamma\w M^\varphi$, and so
\eqref{putin} follows from Proposition~\ref{gata2} and \eqref{smal}.  In the same way
\eqref{bgysin} is a generalization to $\B(X)$ of $i_*i^*$.

\begin{ex}\label{bgysin1}
If $Z_\J$ is a divisor, i.e., $\kappa=1$, then we can assume that $\varphi$  is a section of a line bundle
$L\to X$.   Then $N_\J X=L|_Z$, cf.~\eqref{polly2}.
Assume that $\mu\in\GZ_k(X)$ is irreducible.
If  $\varphi$ vanishes identically on $\mu$, then  $M^\varphi\w\mu=\mu$, and hence
$
(\hat c(L)\w M^\varphi\w\mu )_{k-1}=\hat c_1(L)\w\mu.
$  
Otherwise 
$M^\varphi_0\w\mu=0$ and then 
\begin{equation}\label{krut}
 (\hat c(L)\w M^\varphi\w \mu)_{k-1}= M_1^\varphi\w \mu.
\end{equation}
\end{ex}

\section{Variants of the Poincar\'e-Lelong formula}\label{plvariants}

Let $h$ be a meromorphic section of a Hermitian line bundle $L\to X$.   
We say the  $\div h$ {\it intersects $\mu\in\GZ_k(X)$ properly} if for each irreducible component  $\mu_j$   of $\mu$, 
$\div h$ intersects $|\mu_j|$ properly, cf.~Section~\ref{prel}, i.e., $h$ is non-trivial on each $|\mu_j|$. 
We have the following  Poincar\'e-Lelong formula ``on $\mu$'':

\begin{prop}\label{myrstack}
Assume that $h$ is a meromorphic section of $L\to X$ such that $\div h$ intersects  $\mu\in \GZ_k(X)$ properly.
Then $\log|h|^2\cdot \mu$, a priori defined where $h$ is holomorphic and non-zero, extends to a
current of order $0$ on $|\mu|$. 
Moreover, there is a generalized cycle $[\div h]\w\mu$ in $\GZ_{k-1}(X)$ with Zariski support on $|\div h|\cap|\mu|$
such that 
\begin{equation}\label{mysnitt}
dd^c\big(\log|h|^2 \cdot \mu\big)= [\div h]\wedge \mu - \hat c _1(L)\w\mu.
\end{equation}
If $\mu\sim 0$, then $[\div h]\w\mu\sim 0$.
\end{prop}

We say that  $[\div h]\w\mu$ is the {\it proper intersection} of $\div h$ and $\mu$. 
Choosing a trivial metric on $L$ locally, we see that $[\div h]\w\mu$
 only depends on the divisor
$\div h$ and not on $h$ (since $dd^c (u\cdot \mu)=0$ if $u$ is pluriharmonic). In view of the
last statement of the proposition, $[\div h]\w\mu$ is well-defined in $\B_{k-1}(X)$ for $\mu\in\B_k(X)$.

\begin{proof}
By assumption, $\log|h|^2\cdot \mu$ is generically defined on $|\mu|$.
Each irreducible component $\mu_j$ of $\mu$ is a
finite sum of non-zero generalized cycles $\mu'=\tau_*\alpha$ with $\tau(W)=V:=|\mu_j|$,
see Remark~\ref{planta} and Lemma~\ref{gastub}.
Let us consider such a $\mu'$
and let $V'\subset V$ be a subset of positive codimension such that $h$ is holomorphic and non-vanishing
on $V\setminus V'$. Then 
\begin{equation}\label{anita1}
\log|h|^2\cdot \mu'=\tau_*(\log|\tau^*h|^2\cdot\alpha)
\end{equation}
holds there, and since the right hand side has an extension to $V$ of order $0$ so has
the left hand side. Since  $\tau^{-1}V'$ has positive codimension in $W$, 
$\1_{V'}\tau_*(\log|\tau^*|^2\cdot\alpha)=\tau_*(\1_{\tau^{-1}V'}(\log|\tau^*|^2\cdot\alpha))=0$.
Summing up the first claim of the proposition follows. 

Consider again a $\mu'=\tau_*\alpha$ as above. %
From the usual Poincar\'e-Lelong formula, cf.\ Proposition~\ref{plf}, we have 
\begin{equation}\label{myrslok}
dd^c\big(\log|\tau^*h|^2_{}\cdot \alpha\big)=[\div\tau^*h]\w\alpha-\hat c_1(\tau^*L)\w\alpha
\end{equation}
on $W$.   
Summing up we get \eqref{mysnitt} with
$[\div h ]\w \mu$ defined as the sum of all  $\tau_*([\div \tau^* h]\w\alpha)$.
The last statement of the proposition follows since  $\tau_*([\div \tau^* h]\w\alpha)\sim 0$ if $\alpha$ is of the form
$\beta\w\alpha'$, where $\beta$ is a component of a $B$-form.
\end{proof}

If $h$ is holomorphic,
$
M^h_1\w\mu=\tau_* (M^{\tau^*h}_1\w\alpha)=\tau_*([\div\tau^*h]\w\alpha)
$  
and thus,
cf.~\eqref{krut},  
\begin{equation}\label{propdef}
[\div h]\w\mu = M^h_1\w\mu.
\end{equation}
It follows as in the proof of Theorem~\ref{putig} that 
$\log|h|^2\cdot \mu =\lim_\epsilon\log(|h|^2+\epsilon)\cdot \mu$.

 \smallskip
Now assume that $h=(h_1,\ldots,h_m)$ is a tuple of global sections of $L$ and consider the section
$h$ of $\oplus_1^m L$.  In view of \eqref{spott}   we have
\begin{equation}\label{harpa}
dd^c\big(\log|h|^2\cdot\mu\big)=dd^c\log|h|^2\w\mu,
\end{equation}
where the right hand side is defined by the limit procedure in \eqref{spott}. If $e$ is a local frame for $L$,
then $h=h(e) e$, where $h(e)$ is a tuple of holomorphic functions. Clearly $\log|h(e)|^2$ depends on the choice of
frame but $dd^c\log|h(e)|^2$ does not. Thus 
$$
dd^c\log|h|^2_\s\w\mu := dd^c\log|h(e)|^2\wedge\mu
$$ 
is a well-defined global current which in addition
is independent of the Hermitian metric on $L$.

\begin{remark}\label{rotsak}
Let $U\subset X$ be an open set where we have a local frame $e$ for $L$.  If we choose the
metric on $L$ in $U$ so that $|e|=1$ and equip $E=\oplus_0^m L$ with the induced metric, then
$dd^c\log|\fff|^2_\s\w\mu=dd^c\log|\fff|^2\w\mu$.
\end{remark}

For instance, if $h$ is just one single section, i.e., $m=1$, then 
\eqref{mysnitt} implies that 
$dd^c\big(\log|h|_\s^2\cdot\mu\big)=[\div h]\w\mu.$

\begin{ex}
Let $\theta=dd^c\log (|z_1|^2+|z_2|^2)$ be the generalized cycle in $\mathbb{P}^2$ of Example~\ref{planmedel} 
and let $\fff$ be a section of
$\mathcal{O}(1)$ defining a line through $p=[1:0:0]$. Then $\theta$ has dimension $1$, 
it is irreducible and $|\theta|=\mathbb{P}^2$. Thus
$\div \fff$ intersects $\theta$ properly. We claim that  $[\div \fff]\wedge \theta=[p]$. 
Let $i\colon V\hookrightarrow\Pk^2$ be the line $\div \fff$. Notice that if we consider $z_j$ as sections of
the line bundle $\Ok(1)\to\Pk^2$, then  
$\theta=dd^c\log|h|^2_\s$,  where  $h=(z_1,z_2)$.  Now
$$
[\div \fff]\w\theta=\theta\w[\div \fff]=dd^c\big(\log|h|_\s^2\cdot[V]\big)=i_*dd^c\log|i^*h|_\s,
$$
where the first equality follows from \cite[Ch.~III~Corollary~4.11]{Dem2}
and the second one from \eqref{harpa}. 
In the affinization
where $z_0=1$ we have the frame element $e=z_0$, so in local coordinates $(z_1,z_2)$ we have 
$\log|h(e)|^2=\log (|z_1|^2+|z_2|^2)$; notice that it is 
harmonic on $V\setminus\{p\}$
and has a simple pole at $p$ so that $dd^c\log|i^*h|^2_\s=[p].$
Now the claim follows since $i_*[p]=[p]$.
Notice that $\text{dim}\,|\theta|=2$ while
$\text{dim}\,|[\div h]\wedge \theta|=0$. 
\end{ex}

\section{The $\B$-St\"uckrad-Vogel class}\label{vogelsec} 
Throughout this section $X$ is a compact (reduced) analytic space and 
$\J\to X$ is generated by a finite number of global sections of the line bundle $L\to X$,
to begin with without any specified Hermitian metric. 
For instance, if $X$ is projective, then given $\J\to X$  there is a very ample $L\to X$ such that $L\otimes \J$
is globally finitely generated, see, e.g., \cite[Theorem 1.2.6]{Laz}.  

The classical St\"{u}ckrad-Vogel (SV) algorithm, \cite{SV}, is a way to produce  intersections
by reducing to proper intersections of cycles by divisors. 
The resulting SV-cycles define an element, the SV-class $v(\J,L,X)$,  in $\A(Z)$ that only depends on $\J$ and
the line bundle $L$.  It is related to the Segre class $s(\J,X)$ via van~Gastel's formulas, \cite{gast}, see below.

We shall define an analogous $\B$-SV class $V(\J,L,\mu)$ in $\B(|\mu|\cap Z)$ for any $\mu\in\B(X)$,
and this class will be related to our Segre class $S(\J,\mu)$ via analogues of van~Gastel's
formulas. To motivate our definition we first consider the SV-algorithm on a generalized cycle $\mu\in\GZ_d(X)$:
If $\mu_0:=\1_{X\setminus Z} \mu =0$ then $\J$ vanishes identically on $\mu$ and the algorithm stops directly.
Otherwise, let  $\mu^1,\ldots,\mu^m$ be the irreducible components of $\mu_0$. These
are precisely the irreducible components of $\mu$ that are not contained in $Z$.  
For each $j$ the set of $h\in \Gamma(X, L\otimes \J)$ that 
vanish identically on $\mu^j$ is a proper subspace $V^j$ of the finite-dimensional vector space
$\Gamma(X, L\otimes \J)$.   Thus
each  $h\in \Gamma(X, L\otimes \J)$ in the complement of $\cup_j V^j$ intersects
$\mu_0$ properly, by definition; that is, a {\it generic}  $h$ will do. 
 Let us choose such a section and call it $h_1$.
Next, we consider $\mu_1:=\1_{X\setminus Z} [\div h_1]\w\mu_0 $. 
If $\mu_1$ is empty the algorithm stops.  If not, a generic $h$ intersects 
$\mu_1$     properly. Let us choose such a section and call it $h_2$. 
We proceed
 in this way until 
$\mu_{k}=0$ for some $k\le d$ and the algorithm stops. If 
 $\mu_d:=\1_{X\setminus Z} [\div h_d]\w \mu_{d-1}$ is nonempty, then
 since  $\mu_d$  has dimension $0$, any proper intersection with a divisor $\div h$ will give just $0$, and the
SV-algorithm stops. 

If $\mu_k=0$ for some $k<d$, then we can choose $h_{k+1}, \ldots, h_d$ in an arbitrary way
if we adopt the convention that any $\div h$ intersects the generalized cycle $0$ properly and 
$[\div h]\cdot 0=0$.  We have the following definitions:

\begin{df}
An ordered 
sequence $h=(h_1,h_2,\ldots, h_d)$  of sections of $L\otimes\J$ is a
{\it St\"uckrad-Vogel (SV) sequence on} $\mu\in\GZ_d(X)$ if
$\div h_k$ intersects 
\begin{equation}\label{pluto}
\1_{X\setminus Z}[\div h_{k-1}]\wedge\cdots\wedge\1_{X\setminus Z}[\div h_1]\wedge\1_{X\setminus Z}\mu
\end{equation}
properly,  $k=1,\ldots,d$.
Given a SV-sequence $h$ on $\mu$,
we have the associated\footnote{It would me more correct but somewhat inconvenient to use the
term SV-generalized cycle.}
 \emph{SV-cycle}  
\begin{equation}\label{sv}
v^h\wedge\mu =\sum_{k=0}^d v_k^h\w\mu, 
\end{equation}
where 
$$
v_0^h\w\mu :=\1_Z\mu, \ \
v_k^h\w\mu:=\1_Z[\div h_{k}]\wedge\1_{X\setminus Z}[\div h_{k-1}]\wedge\cdots\wedge\1_{X\setminus Z}[\div h_1]\wedge\1_{X\setminus Z}\mu, \  k\ge 1.
$$
\end{df}
Here we use the convention that $\1_V$ acts on the whole current on its right, i.e., $\1_V a\w b\w \ldots= \1_V(a\w b\w\ldots)$,
cf.~\cite[Sections~3 and 6]{aswy}.  

\begin{ex} 
If $\J$ vanishes identically on $\mu$,  then $v^h=v^h_0=\mu$.  
\end{ex}

If $\div h$ does not intersect $\mu$ properly we can still define  $[\div h]\wedge\mu$ by \eqref{propdef}.
Since $M^h_1\w\mu'=0$ if $h$ vanishes identically on $|\mu'|$ 
we have that 
$
[\div h]\wedge\mu=\sum_j [\div h]\wedge \mu'_j,
$
where $\mu'_j$ are the irreducible components of $\mu$ that $\div h$ intersects properly. 
By this convention therefore
$[\div h]\wedge\mu=[\div h]\wedge\1_{X\setminus Z}\mu$ 
if 
the right hand side is a proper intersection. 
For any sequence
$h=(h_1,\ldots,h_d)$ of sections of $L\otimes \J$ we can thus define  \eqref{sv}     with
\begin{equation}\label{polyp}
v^h_0\w\mu=\1_Z\mu, \quad   v^h_k\wedge\mu={\bf 1}_Z[\div h_k]\w\mu\cdots \w [\div h_1]\wedge \mu,
\quad k=1,2,\ldots d,
\end{equation}
and as long as $h$ is a SV-sequence it is consistent with the previous definition. 
\smallskip
Let $\fff=(\fff_0,\ldots,\fff_m)$ be a sequence of global sections of $L$
that generate $\J$.  
Given $a\in \Pk^m$, $a=[a_0: \cdots:a_m]$,  let $a\cdot \fff=a_0\fff_0+\cdots+a_m\fff_m$ which
is well-defined up to a nonzero constant.  
If  $a=(a_1,\ldots,a_d)\in(\mathbb{P}^m)^d$ is a generic tuple, then $a\cdot \fff=(a_1\cdot \fff, a_2\cdot \fff, \ldots, a_d\cdot \fff)$
is a SV-sequence on $\mu$ and, cf.~\eqref{polyp}, 
\begin{equation}\label{puhu}
v^{a\cdot \fff}\w\mu=\1_Z\mu+ 
\sum_{k\ge 1}\1_Z [\div (a_k\cdot \fff)]\w\ldots\w[\div (a_1\cdot \fff)]\w \mu 
\end{equation}
is the associated SV-cycle.  As observed above,  however,  \eqref{puhu} makes sense for
any $(a_1,\ldots,a_d)\in(\mathbb{P}^m)^d$.

\begin{prop} \label{pilt0} 
Assume that $\mu\in\GZ(X)$.
Then 
\begin{equation}\label{sork}
\int_{ (\Pk^m)^k} [\div (a_k\cdot \fff)]\w\ldots\w[\div (a_1\cdot \fff)]\w \mu \, dV(a) =
(dd^c\log|\fff|^2_\s)^k\w\mu,\quad k=1,2,\ldots,
\end{equation}
where $dV(a)=\wedge_{j=1}^k (dd^c\log|a_j|^2)^m$ is the natural normalized volume form on $(\Pk^m)^k$.
\end{prop}

\begin{proof}
We may assume that $\mu=\tau_*\alpha$  where $\tau\colon W\to X$ and $W$  is connected. Then $\mu$ is irreducible.
If $\J$ vanishes identically on $\mu$, then $\fff\equiv 0$ on $\mu$ and by definition both sides
of \eqref{sork} vanish.   
Thus we can assume that $\tau^*\J$ is nontrivial on $W$, 
$\tau^*\J$ is principal and that the exceptional divisor $D$ is defined by the section $\fff^0$ of the line bundle
$L_D\to W$.  Then $\tau^*\fff=\fff^0\fff'$ where $\fff'=(\fff_0',\ldots,\fff_m')$ is a non-vanishing tuple of
sections of $L_D^{-1}\otimes \tau^*L$.   Notice that
$$
dd^c\log|a_j\cdot \tau^*\fff|^2_\s=dd^c\log|\fff^0|^2_\s +dd^c\log|a_j\cdot \fff'|^2_\s=[D]+[\div (a_j\cdot \fff')].
$$
Since $\fff'$ is non-vanishing on $|D|$, as in \cite[Eq.~(6.3)]{aswy},  
for generic  $a_j$ we have   
\begin{multline*}
 [\div (a_k\cdot \fff)]\w\ldots\w[\div (a_1\cdot \fff)]\w \mu =\\
 \tau_*\big([D]\w[\div(a_{k-1}\cdot \fff')]\w \cdots\w[\div(a_1\cdot \fff')]\w\alpha + [\div(a_{k}\cdot \fff')]\w 
 \cdots\w[\div(a_1\cdot \fff')]\w\alpha \big),
 \end{multline*}
where all intersections are proper. 
By \cite[Lemma~6.3]{aswy} the left hand side of \eqref{sork}  is therefore equal to
\begin{equation}\label{samband}
 \tau_*\big([D]\w (dd^c\log|\fff'|^2_\s)^{k-1}\w\alpha+(dd^c\log|\fff'|^2_\s)^k \w\alpha \big).
\end{equation} 
Now assume that we locally have a flat metric on $L$.  
With the notation in the proof of Theorem~\ref{putig} we then have 
$dd^c\log|\fff'|_\s^2=\hat\omega$ since $\fff'$ is a tuple of sections of
$L_D^{-1}$ and $|\fff'|=1$, cf.~\eqref{pilla} below. 
From  \eqref{rot} we can therefore deduce that 
\eqref{samband} is equal to the right hand side of \eqref{sork}. 
\end{proof}

\begin{remark}
Proposition~\ref{pilt0} is similar to \cite[Theorem~6.2]{aswy}. The analogues of the identities (6.8) and (6.9) in that
theorem hold in the present situation as well; after adaption to the present situation the proof in \cite{aswy} goes through. 
\end{remark}

\begin{df}
Given a line bundle $L$ and a tuple $\fff=(\fff_0,\ldots, \fff_m)$ of sections of $L$ 
that generate $\J$  and $\mu\in\GZ(X)$ we define the generalized cycle
$$
M^{L, \fff}\w\mu=\1_Z\mu+\sum_{k\ge 1} \1_Z (dd^c\log|\fff|^2_\s)^k\w \mu.
$$
\end{df}

\noindent It follows from   \eqref{puhu} and \eqref{sork} that  $M^{L, \fff}\w\mu$ is a mean value of 
SV-cycles on $\mu$.  

\smallskip
Notice that in general, the subspace of $\Gamma(X, L)$ generated by $\fff_0, \dots,\fff_m$ is
proper.  Nevertheless,  the class of $M^{L,\fff}\wedge\mu$ in $\B(X)$ is independent of the choice of tuple $\fff$:

\begin{prop}\label{pilt}
If $g$ is another tuple of sections of $L$ that generate $\J$, $\mu,\mu'\in\GZ(X)$ and $\mu'\sim \mu$, then 
$
M^{L,\fff}\w\mu\sim M^{L,g}\w\mu'.
$
\end{prop}

\begin{proof}
We first consider $\mu$ and keep the notation from the proof of Proposition~\ref{pilt0}.   
If $\J$ vanishes on $|\mu|$, then $M^{L,\fff}\w\mu= M^{L,\fff}_0\w\mu=\mu$. Thus we assume that
$\tau^*\fff=\fff^0\fff'$ on $W$ as usual. 
Since $\fff'$ is a non-vanishing  tuple of sections
of $L_D^{-1}\otimes\tau^*L$,  $1/|\fff'|^2_\s$  is a metric on $L_D^{-1}\otimes\tau^*L$
and hence $dd^c\log|\fff'|^2_\s$ is a representative of the first Chern class
$c_1(L_D^{-1}\otimes\tau^*L)$.
From \eqref{samband} and \eqref{linalg}   we have
\begin{equation}\label{pilt2}
M^{L,\fff}_k\w\mu=\1_Z (dd^c\log|\fff|^2_\s)^k\w\mu=\tau_*\big([D]\w (dd^c\log|\fff'|^2_\s)^{k-1}\w\alpha\big), \quad k=1,2,\ldots.
\end{equation}
Now $\tau^*g=\fff^0 g'$, where $g'$ also is a tuple of sections of $L_D^{-1}\otimes\tau^*L$, and hence
$dd^c\log|g'|^2_\s$ and $dd^c\log|\fff'|^2_\s$  differ by a $B$-form;  in fact the difference is 
$dd^c$  of the global function $\log(|g'|^2_\s/|\fff'|^2_\s)$.  Thus the class in $\B(X)$ is independent
of the choice of tuple. Finally, if $\beta$ is a component of a $B$-form and $\alpha=\beta\w\alpha'$, then 
\eqref{pilt2} is $\sim 0$ in $\GZ(X)$.  
\end{proof}

In view of Remark~\ref{rotsak} we have that 
$M^{L,\fff}\w\mu=M^\fff\w\mu$ in $U\subset X$ for a suitable metric on $L$ on $U$.
Therefore local statements that hold for $M^\fff\w\mu$ 
must hold for $M^{L,\fff}\w\mu$ as well:
For instance, if $\gamma$ is a  component of a Chern or Segre form,
then by \eqref{samba2}, 
\begin{equation}\label{glattis}
M^{L,\fff}\w(\gamma\w\mu)=\gamma\w M^{L,\fff}\w\mu.
\end{equation}
If $h\colon X'\to X$ is proper and $\fff$ is a tuple of sections of $L$ that generate
$\J\to X$, then $h^*\fff$ is a tuple of sections of $h^*L$ that generate $h^*\J\to X'$.  If $\mu\in\GZ(X')$, then
by Theorem~\ref{putig}~(iv), 
\begin{equation}\label{plastat}
h_* \big(M^{h^*L, h^*\fff}\w\mu'\big)=M^{L,\fff}\w h_*\mu'.
\end{equation}

\smallskip

In view of Proposition~\ref{pilt} the following definition makes sense.

\begin{df}
Assume that $L\to X$ has sections that generate $\J$ globally. For 
$\mu\in\B(X)$ we let  $V(\J,L,\mu)$, the {\it $\B$-SV class},
be the class in $\B(Z\cap|\mu|)$ defined by 
 $M^{L,\fff}\w\hat\mu$ for a tuple of generators $\fff$ and a representative $\hat\mu\in \GZ(X)$  of
$\mu$.
\end{df}

\smallskip

We now relate the $\B$-SV class to the $\B$-Segre class in analogy with van~Gastel's formulas \cite{gast}, see \eqref{gastel1} below.
To this end we first give a $\GZ$-variant and therefore choose a Hermitian metric.

\begin{thm}\label{Bgastel} 
Let $\fff=(\fff_0,\ldots,\fff_m)$ be a tuple of sections of $L$ that generate $\J$.  Assume that we have a Hermitian
metric on $L$ with first Chern  form $\hat\omega_L$
and consider $\fff$ as a section of the Hermitian vector bundle $E=\oplus_0^m L$. 
For $\mu\in\GZ(X)$ we have 
\begin{equation}\label{gastel12}
M^{L,\fff}\w\mu=\sum_{j\ge 0}\Big(\frac{1}{1-\hat\omega_L}\Big)^j\w M^{\fff}_j\w\mu
\end{equation}
and
\begin{equation}\label{gastel22}
M^{\fff}\w\mu=\sum_{j\ge 0}\Big(\frac{1}{1+\hat\omega_L}\Big)^j\w M^{L,\fff}_j
\w\mu.
\end{equation}
\end{thm}

\begin{proof}
Let us assume that $\mu=\tau_*\alpha$ where $\tau\colon W\to X$ and $W$ is connected.  
If $\J$ vanishes identically on $\mu$ then $M^\fff\w\mu=M^\fff_0\w\mu=\mu$ and
$M^{L,\fff}\w\mu=M^{L,\fff}_0\w\mu=\mu$ and thus \eqref{gastel12} and \eqref{gastel22}
are both trivially true. 

We can thus assume that $\tau^*\fff=\fff^0\fff'$,  where $\fff'$ is a non-vanishing tuple of sections
of $L_D^{-1}\otimes \tau^* L$,    or equivalently, a non-vanishing section of $\tau^*E$. 
As in the proof of Theorem~\ref{putig}  we let $\hat\omega=-\hat c_1(L_D)$. 
Let $e^{-\psi}$ be the induced metric on $L_D^{-1}\otimes\tau^*L$. Then, cf.~the proof of Theorem~\ref{putig}, 
$1=|\fff'|^2=|\fff'|_\s e^{-\psi}$ so that $|\fff'|^2_\s=e^{\psi}$ and hence
\begin{equation}\label{pilla}
dd^c\log|\fff'|^2_\s=dd^c\psi=\hat\omega+\tau^*\hat\omega_L.
\end{equation}
It follows from \eqref{pilt2} and \eqref{pilla} 
that\footnote{Notice that $\hat\omega+\tau^*\hat\omega_L$ is independent of the
metric on $L$.}
\begin{equation}\label{buksvin}
M^{L,\fff}\wedge\mu=\tau_*\big([D]\w\frac{1}{1-\hat\omega-\tau^*\hat\omega_L}\wedge\alpha\big).
\end{equation}
 We have, cf.\ \eqref{trott},  
\begin{multline*}
\sum_{j\ge 0}\Big(\frac{1}{1-\hat\omega_L}\Big)^j\w M^\fff_j\w\mu=
\sum_{j\ge 0}\Big(\frac{1}{1-\hat\omega_L}\Big)^{j+1}
\w \tau_*\big([D]\w\hat\omega^j\w\alpha)=\\
\frac{1}{1-\hat\omega_L}\w \tau_*\Big([D]\w\sum_{j\ge 0}
\Big(\frac{\hat\omega}{1-\tau^*\hat\omega_L}\Big)^j\w\alpha\Big)=\\
\frac{1}{1-\hat\omega_L}\w\tau_*\Big([D]\w\frac{1-\tau^*\hat\omega_L}{1-\tau^*\hat\omega_L-\hat\omega}\w\alpha\Big)=
M^{L,\fff}\w\mu.
\end{multline*}
Thus \eqref{gastel12} follows,  and 
\eqref{gastel22} is proved in a similar way. 
\end{proof}

\begin{cor}
If $\mu\in\B_k(X)$ and $\omega_L=c_1(L)$ we have 
\begin{equation}\label{gastel100}
V(\J,L,\mu) =\sum_{j\ge 0}\Big(\frac{1}{1-\omega_L}\Big)^j\w S_j(\J,\mu), \ \
S(\J,\mu)=\sum_{j\ge 0}\Big(\frac{1}{1+\omega_L}\Big)^j\w V_j(\J,L,\mu).
\end{equation}
\end{cor}

\begin{remark}\label{fixmov}
Suppose that $\mu=\1_X$. 
Fix  $k\ge 0$ and consider the decomposition, cf.~\eqref{deko},
$
 M_k^{L,\fff}=M_{k,fix}^{L,\fff} + M_{k,mov}^{L,\fff}
$
of the component $M_k^{L,\fff}$ of codimension $k$.
Since $M_k^{L,\fff}$ is obtained as a mean value of 
$v_k^{a\cdot \fff}$, which are cycles of pure codimension $k$,
it is clear that any irreducible $k$-cycle $V$ of $X$ that occurs in all generic SV-cycles must
appear in $M_{k,fix}^{L,\fff}$. In the literature, such a cycle $V$ is called a fixed component;
any other component in a generic  SV-cycle is called a moving component.
Since the Zariski support of the irreducible components
of $M_{k,mov}^{L,\fff}$ have codimension strictly smaller than $k$,
they must be mean values of moving components of
$v_k^{a\cdot \fff}$.
It follows from  \eqref{gastel100} 
that  the fixed components of
$V_{k}(\J,L,X)$ and $S_{k}(\J,X)$ are the same, cf.\ \eqref{deko2}.    
\end{remark}

\begin{ex}\label{vogelpn}  
Assume that $X=\P^{d+\kappa}_{[z_0:\cdots :z_d:w_1:\cdots :w_\kappa]}$, 
let $L:=\Ok(1)$, and let  $i\colon Z\to X$  be  the linear subspace $\{w_1=\ldots=w_\kappa=0\}$.
If $\J=\J(w_1,\ldots,w_\kappa)$, then 
$Z=Z_\J$ is a smooth (hence regular) embedding defined by the section
$w$ of $E=\oplus_1^\kappa L$.  Thus $E|_Z\simeq N_\J X$,
 cf.~\eqref{polly2}. We want to compute the $\B$-Gysin mapping \eqref{bgysin},
 or more precisely \eqref{jupiter},
in this case.
If we equip $L$ with the Fubini-Study  
metric $\hat\omega_L=dd^c\log(|z|^2+|w|^2)_\circ$, then 
$\hat c(L)=1+\hat\omega_L$,  and  therefore
\begin{equation}\label{kokett}
\hat c(N_\J X)=(1+\hat\omega_L)^\kappa.
\end{equation}
Let $\mu\in\GZ_k(X)$. 
By \eqref{gastel22} 
\begin{multline}\label{stork}
\big(\hat c(N_\J X)\w M^w\w\mu)_{k-\kappa}= \big((1+\hat{\omega}_L)^\kappa\w M^w\w\mu\big)_{k-\kappa}=\\
\big(\sum_{j\ge 0}(1+\hat{\omega}_L)^{\kappa-j}\w M^{L,w}_j\w\mu\big)_{k-\kappa}
=\sum_{j\ge 0}\hat{\omega}^{\kappa-j}_L\w  M^{L,w}_j\w\mu. 
\end{multline}  
By \eqref{stork} one can thus reduce the computation of \eqref{jupiter} 
to find  $M^{L,w}_j\w\mu$,  which in turn can be obtained as mean values of 
generic SV-cycles.  
\end{ex}

From  Theorem~\ref{Bgastel}  and Lemma~\ref{spe} we have

\begin{prop}\label{proppfull}
For  each $x\in X$, 
$
\mult_x (M_k^{L,\fff}\wedge\mu)=\mult_x (M_k^{\fff}\wedge\mu).    
$
\end{prop}


For $\mu\in \GZ_d(X)$ we let  
$$
\deg_L\mu:=\int_X \hat\omega_L^d\wedge\mu,
$$
where $\hat\omega_L$ is any representative of $c_1(L)$; 
by Stokes' theorem it is well-defined.
Moreover, in view of \eqref{kusin} it only depends on the image of $\mu$ in $\B_d(X)$
and so $\deg_L$ is well-defined on $\B(X)$.
If $\mu$ is  a cycle,
then $\deg_L\mu$ is the usual degree of $\mu$ with respect to $L$.
The degree is indeed the mass with respect to $L$ of $\mu$, and 
we  have the following mass formula:

\begin{prop}
If  $\J\to X$ is generated by the  tuple $\fff=(\fff_0,\ldots,\fff_m)$ of sections of   $L\to X$ and 
$\mu\in\GZ_d(X)$, then 
\begin{equation}\label{bezout2}
\deg_L \mu=\deg_L(M^{L,\fff}_0\wedge\mu)+\cdots +\deg_L (M^{L,\fff}_d\wedge\mu) +\deg_L (\ett_{X\setminus Z}(dd^c\log|\fff|^2_\s)^d\w\mu).
\end{equation}
If $m+1\le d$, then the last term on the right hand side 
vanishes.
\end{prop}

\begin{proof}
We can assume that $\mu=\tau_*\alpha$ where $\tau\colon W\to X$   and
$W$   is connected. If $\J$ vanishes identically on $\mu$, then both sides of \eqref{bezout2} are
equal to $\deg_L\mu$.  Otherwise we may assume that $\tau^*\fff=\fff^0\fff'$ where $\fff^0$ is a section of 
the line bundle $L_D$ 
defining the divisor $D$ on $W$ and $\fff'$ is a non-vanishing section of $\tau^*E\otimes L^{-1}_D$, where $E=\oplus_0^m L$. 
Notice that in view of \eqref{ingen},
$$
\deg_L \mu=\int_W\tau^*\hat\omega^d_L\wedge\alpha.  
$$
Let $\hat\omega=-\hat{c}_1(L_D)$. By \eqref{pilla}, 
$\omega_\fff:=dd^c\log |\fff'|^2_\s =\hat\omega+\tau^*\hat\omega_L$. 
From \eqref{doris} thus   
\begin{equation}\label{gurka}
\tau^*\hat\omega_L=[D]+ \omega_\fff +dd^c\nu,
\end{equation}
where $\nu=-\log|\tau^*\fff|$, which is a a global integrable form on $W$.  
By repeated use of Stokes' theorem we get
\begin{multline*}
\int_{W}\tau^*\hat\omega^d_L\wedge\alpha=\int_{W}\tau^*\hat\omega^{d-1}_L\wedge [D]\w\alpha+
\int_{W} \tau^*\hat\omega^{d-1}_L\wedge \omega_\fff\w\alpha=\\
\int_{W}\tau^*\hat\omega^{d-1}_L\wedge [D]\w\alpha+
\int_{W} \tau^*\hat\omega^{d-2}_L\w [D]\w \omega_\fff\w\alpha+
\int_{W}\tau^*\hat\omega^{d-2}_L\w \omega_\fff^2\w\alpha=\\
\int_{W}\tau^*\hat\omega^{d-1}_L\w [D]\w\alpha+
\int_{W}\tau^*\hat\omega^{d-2}_L\w\omega_\fff\w [D]\w \alpha+\cdots +
\int_{W}\omega_\fff^{d-1}\w [D]\w \alpha+
\int_{W}\omega_\fff^d\w\alpha.
\end{multline*}
Now \eqref{bezout2} follows from the proof of Proposition~\ref{pilt0}, cf.~\eqref{samband} and \eqref{pilt2},
since  
$$
\1_{X\setminus Z}(dd^c\log|\fff|_\s^2)^d\w\mu=\tau_*(\omega_\fff^d\w\alpha).
$$
The last statement follows since  $(dd^c\log|\fff'|_\s^2)^{m+1}=0$.  
\end{proof}

If $u$ is  $\hat\omega_L$-plurisubharmonic with analytic singularities, then one can define
$(dd^cu)^k$ for any $k$ and an analogous mass formula was proved in \cite{abw},
see \cite[Theorem~1.2]{abw}.

\begin{remark}
Assume that $\fff$ and $\mu=\tau_*\alpha$ are as in the previous proof. If $g$ is a section of $L\otimes\J$, then
$\tau^*g=\fff^0g'$ where $g'$ is a section of $L_D^{-1}\otimes\tau^*L$.  Let  $h_1,\ldots,h_d$ be a
SV-sequence on $\mu$ and $v^h\w\mu$ be the associated SV-cycle. If $h$ is sufficiently generic, then
with essentially the same proof we get
$$
\deg_L \mu=\deg_L (v_0^h\w\mu)+\cdots +\deg_L (v_d^h\w\mu) +
\deg_L (\ett_{X\setminus Z}[\div h_d]\w\ldots\w[\div h_1]\w\mu).
$$
\end{remark}

Finally let us consider the special case when $\mu$ is an ordinary cycle. With no loss of
generality we can assume that $\mu=\1_X$.  
Let $h=(h_1,h_2,\ldots, h_n)$ be a 
sequence of sections of $L\otimes\J$.  One can check that $h$  is an SV-sequence on $X$ if and only if
\begin{equation}\label{sture}
\codim\big( (X\setminus Z)\cap\{h_1=\cdots =h_k=0\}\big)= k\ {\rm or}\ \infty,\quad k=1,\ldots, n;
\end{equation}
this is the condition in \cite{SV}.  The SV-algorithm in \cite{SV} is precisely the same as used above
and the resulting SV-cycle therefore is, in our notation, cf.\ \eqref{polyp},
$$
v^h=\1_Z+\1_Z[\div h_1]+\cdots +\1_Z[\div h_n]\w\ldots \w [\div h_1].
$$

Let us  now assume that $X$ is irreducible.  If $\J$ vanishes identically, then $v^h=\1_X$ for any
SV-sequence, and we define
$v(\J,L,X)=v_0(\J,L,X)=1$.
Otherwise, let us 
assume that $\tau\colon X'\to X$ is a modification such that $\tau^*\J$ is principal, and let $D$ and $L_D$
be as before. In particular, let $\fff^0$ be a section of $L_D$ that defines the divisor $D$. 
Then $\tau^*h_k=\fff^0 h'_k$ where $h_k'$ are sections of $L_D^{-1}\otimes\tau^*L$. As in the proof of 
Proposition~\ref{pilt0}, cf.~\cite[Eq.~(6.3)]{aswy},  
we then have  
$$
v^h_k=\tau_*\big([D]\w[\div h_{k-1}']\w\ldots \w[\div h_1']\big),
$$ 
where the case $k=1$ shall be interpreted as $\tau_*[D]$. 
Choosing the sequence $h_j$ even more generic if necessary, we can in addition assume that
all the intersections  
\begin{equation}\label{hapy}
[\div h'_{k-1}]\w\ldots\w[\div h'_1]\w[D]
\end{equation} 
are proper. 
As before, let $\omega_L:=c_1(L)$ and $\omega=-c_1(L_D)$.
Then the first Chern class of $L^{-1}_D\otimes \tau^* L$ is
$\omega+\tau^*\omega_L$.  
By definition, cf.~Section~\ref{prel}, therefore \eqref{hapy} is a representative of the Chow class
$(\omega+\tau^*\omega_L)^{k-1}\cap[D]$.    
We conclude that a generic SV-sequence defines the Chow class
\begin{equation}\label{snark} 
v(\J, L,X):= \sum_{k\ge
  1}\tau_*\big((\omega+\tau^*\omega_L)^{k-1}\cap [D]\big)=
\tau_*\Big(\frac{1}{1-\omega-\tau^*\omega_L}\cap [D]\Big).
\end{equation}
It follows that this class only depends on $L$ and $\J$ but not on the choice of 
modification of $X$.  
If $X$ is not irreducible and consists of the irreducible components $X^1, X^2,\ldots$, then
we define $v(\J,L,X)=v(\J,L,X^1)+v(\J,L,X^2)+\cdots$.  
The formulas
\begin{equation}\label{gastel1}
v(\J,L,X) =\sum_{j\ge 0}\Big(\frac{1}{1-\omega_L}\Big)^j\cap s_j(\J,X), \quad
s(\J,X)=\sum_{j\ge 0}\Big(\frac{1}{1+\omega_L}\Big)^j\cap v_j(\J,L,X),
\end{equation}
are due to van~Gastel, \cite[Corollary~3.7]{gast}, and 
can be obtained by mimicking the proof of Theorem~\ref{Bgastel} above.

\section{Comparison of $\A(X)$ and $\B(X)$}\label{faksimil}

In this section we assume
that $X$ is compact and projective. In particular,
each line bundle
over $X$ has a nontrivial meromorphic section. 
Let $\widehat H^{k,k}(X)$ be the equivalence classes of 
$d$-closed $(k,k)$-currents $\mu$
on $X$ of order zero such that $\mu\sim 0$ if there is a current
$\gamma$ of order zero such that $\mu=d\gamma$.  
Notice that if $i\colon X\to  Y$ is an embedding into a smooth
manifold $Y$ of dimension $N$, then there is a natural mapping
$i_*\colon \widehat H^{n-k,n-k}(X)\to H^{N-k,N-k}(Y,\C)$ induced by the push-forward of currents.
If $X$ is smooth and $X=Y$, then this map gives an isomorphism
$\widehat H^{n-k,n-k}(X)\simeq H^{n-k,n-k}(X,\C)$; the surjectivity
is clear and the injectivity follows since a closed current of order zero
locally has a potential of order zero.

\begin{ex}\label{dota} 
Assume that $h$ is a meromorphic section of a Hermitian line bundle
$L\to X$ such that $\div h$ intersects $\mu\in\GZ_k(X)$ properly.  It follows from
Proposition~\ref{myrstack} that $[\div h]\w\mu$ and $\hat c_1(L)\w\mu$ coincide in 
$\widehat H^{1,1}(X)$. 
\end{ex}

Let $E\to X$ be a Hermitian vector bundle. Since $\hat c_k(E)$ is smooth and closed on $X$,
$\mu\mapsto \hat c_k(E)\w\mu$ is a well-defined mapping on $\widehat H(X)$.  
Another choice of metric gives rise to a smooth form that is $\hat c_k(E)+dd^c\psi$
for a suitable smooth form $\psi$ (if $k\ge 1$).   Thus  we get a mapping
$\mu\mapsto c_k(E)\w\mu$ on $\widehat H(X)$.  
Let $0\to S\to E\to Q\to 0$ be a short exact sequence of Hermitian vector bundles on $X$.  Then,
cf.~Section~\ref{profet}, there is a smooth $\gamma$ on $X$ such that 
$dd^c\gamma=\hat c(E)-\hat c(S)\w \hat c(Q)$.   Thus   
\begin{equation}\label{zebra}
\big(c(E)-c(S)\w c(Q)\big)\w \mu=0, \quad \mu\in\widehat H(X).
\end{equation}
In view of \eqref{kusin}  there is a natural mapping
$$
B_{X}\colon \B_{k}(X)\to \widehat H^{n-k,n-k}(X), \quad k=0,1,2 \ldots .
$$
If $f\colon X'\to X$ is a proper map, then  
\begin{equation}\label{dot}
B_X f_*\mu=f_*B_{X'}\mu. 
\end{equation}
If $E\to X$ is a vector bundle, then
\begin{equation}\label{puttef}
B_X(c(E)\w\mu)=c(E)\w B_X\mu,  \quad \mu\in\B(X).
\end{equation}

In view of \eqref{chow}  there is a mapping
$$
A_{X}\colon \A_{k}(X)\to \widehat H^{n-k,n-k}(X), \quad k=0,1,2 \ldots,
$$
taking a representative $\hat\mu$ of 
$\mu$  to the cohomology class determined by its Lelong
current. Clearly 
\begin{equation}\label{pumpa}
A_X\mu=B_X\mu, \quad \mu\in\mathcal{Z}(X);
\end{equation}
as a consequence, the image of $A_X$ is contained in the image of $B_X$.
If $f\colon X'\to X$ is proper as above we have from 
\eqref{rapport} that
\begin{equation}\label{dota2}
A_X f_*\mu=f_* A_{X'}\mu, \quad  \mu\in \A(X).
\end{equation}
We will use the   equalities, see \cite[Theorem 3.2]{Fult},
\begin{equation}\label{saga1}
c(E)\cap f_*\mu=f_* (c(f^*E)\cap \mu),
\end{equation}
in $\A(X)$ if $E\to X$ is a vector bundle, and
\begin{equation}\label{saga2}
c(E)\cap \mu=c(L)\cap (c(Q)\cap\mu)
\end{equation}
in $\A(X)$ if $0\to L\to E\to Q\to 0$ is exact.
In analogy with \eqref{puttef} we have:
\begin{equation}\label{bigin}
A_X\big (c(E)\cap \mu\big )=c(E)\wedge A_X\mu.
\end{equation}

\begin{proof}[Proof of Eq.~\eqref{bigin}] 
First assume that $E=L$ has rank $1$; it is then sufficient to show \eqref{bigin}
for $c_1(L)$.   
By $\Z$-linearity it is enough to look at each irreducible
component of $\mu$ separately and so we may assume that $\mu$ is represented in $\A(X)$ by an 
irreducible subvariety $V\hookrightarrow X$. Let $h$ be a meromorphic section of $L$ that is
nontrivial on $V$. 
Then $c_1(L)\cap \mu$ is represented in $\A(X)$ by the cycle $[\div h]\w [V]$ and so
$$
A_X(c_1(L)\cap \mu)=[\div h]\w[V]=  c_1(L)\wedge [V]=c_1(L)\wedge A_X\mu
$$
in $\widehat H(X)$, where the second equality follows from Example~\ref{dota} applied to $[V]$.

Next, assume that \eqref{bigin} holds for vector bundles of rank 
$\leq r$ and consider $E$ of rank $r+1$.  Let $p\colon X'\to X$, where $X'=\P(E)$, 
and let $L=\Ok(-1)$ be the tautological line subbundle
so that  we have a short exact sequence $0\to L\to p^*E\to Q\to 0$ over $X'$. 
Take $\mu'$ in $\A(X')$ such that $p_*\mu'=\mu$.
By \eqref{saga1} and \eqref{saga2},   
\begin{equation}\label{saga12}
c(E)\cap \mu=c(E)\cap p_*\mu'=p_*\big( c(p^* E)\cap \mu'\big) =
p_*\big(c(L)\cap (c(Q)\cap \mu')\big).
\end{equation}
By \eqref{ingen}, \eqref{orgie2}, \eqref{zebra}, \eqref{dota2},  \eqref{saga12}, and the induction hypothesis
\begin{multline*}
A_X (c(E)\cap \mu)= A_X  p_* \big(c(L) \cap (c(Q)\cap \mu')\big)=\\
p_* A_{X'}\big(c(L)\cap (c(Q)\cap\mu') \big)=
p_* \big(c(L)\w A_{X'} (c(Q)\cap\mu') \big)=\\
p_* \big(c(L)\w  c(Q)\w A_{X'}\mu' \big)=
p_*\big(c(p^*E)\w A_{X'}\mu' \big)=
c(E)\w p_*( A_{X'}\mu')= c(E)\w A_X\mu.
\end{multline*}
\end{proof}

\begin{proof}[Proof of Proposition~\ref{ab1}]
We have already noticed, \eqref{pumpa}, that the image of $A_X$ is contained in the image of $B_X$. 
For the converse inclusion consider $\mu=\tau_*\alpha$ in $\GZ_k(X)$, $\tau\colon W\to X$.
By \eqref{dota2},   \eqref{bigin}, \eqref{pumpa}, \eqref{puttef}, and \eqref{dot}
we have 
\begin{multline*}
A_X \tau_* (\alpha\cap \1_W) =\tau_* A_W(\alpha\cap \1_W)=\tau_*(\alpha\w A_W\1_W)=\\
 \tau_*(\alpha\w B_W\1_W)=\tau_*B_W(\alpha\w \1_W)=B_X \tau_*(\alpha\w \1_W)=B_X\mu,
\end{multline*}
and thus $B_X\mu$ is in the image of $A_X$. 
\end{proof}

\begin{proof}[Proof of Proposition~\ref{ab0}]
 Let $N=N_\J X$.
We may assume that $\mu$ is an irreducible subvariety $i\colon V\to X$. If $\J$ vanishes identically on
$V$, then $\mu$ is mapped to $\mu$ under both the Gysin and the $\B$-Gysin mapping. 
Thus we can assume that we have
a modification $\pi\colon V'\to V$ such that
$\pi^*i^*\J$ is principal on $V'$.  Let $D$ be the exceptional divisor and
$L=L_D$ the associated line bundle.
Using \eqref{pumpa}, \eqref{dota2}, and \eqref{bigin}, recalling that $s(\J,V)=s(i^*\J,V)$, 
see the introduction and Remark~\ref{asegre}, we have 
\begin{multline}\label{putte}
A_Z( c(N)\cap s(\J,V))=A_Z( c(N)\cap s(i^*\J,V))=
A_Z\Big(c(N)\cap\pi_*\Big(\frac{1}{c(L)}\cap[D]\Big)\Big)\\
=c(N)\w
A_Z\pi_*\Big(\frac{1}{c(L)}\cap[D]\Big)=
c(N)\w \pi_*A_{|D|}\Big(\frac{1}{c(L)}\cap [D]\Big)\\
=c(N)\w \pi_*\Big(\frac{1}{c(L)}\w A_{|D|}[D]\Big)=
c(N)\w \pi_*\Big(\frac{1}{c(L)}\w B_{|D|}[D]\Big). 
\end{multline}
By an analogous computation backwards with $B_{|D|}$ rather than $A_{|D|}$,
using \eqref{puttef} and \eqref{dot},  we find that the right hand side of 
\eqref{putte} is equal to $B_Z(c(N)\w S(\J,V))$.
\end{proof}

Notice in particular that
$
A_Z s(\J,X)=
B_Z S(\J,X).
$
Summing up we have seen that the $\A$- and $\B$-objects coincide
on cohomology level. However,
there are no nontrivial mappings $T_X\colon \A_k(X)\to \B_k(X)$
such that 
\begin{equation}\label{studs2}
\begin{array}[c]{ccc}
\A_k(X) & \stackrel{{f_*}}{\longrightarrow} & \A_k(Y) \\
\downarrow \scriptstyle{T_X} & &  \downarrow \scriptstyle{T_Y} \\
\B_k(X) & \stackrel{f_*}{\longrightarrow} &  \B_k(Y)
\end{array}
\end{equation}
commutes for each proper mapping mapping $f\colon X\to Y$. 
In fact, let $X$ be a one-point set $\{0\}$, and let $Y$ be a
manifold with two distinct points $p,q$ that are rationally equivalent. 
Take $f$ so that $f(0)=p$. If $T_X$ is nonzero,
then $f_*T_X[0]$ has support at $p$ and is nonzero. If \eqref{studs2}
commutes, then $T_Y[q]=T_Y[p]$ must be a nonzero point mass at $p$.
Changing the roles of $p$ and $q$ we get a contradiction since 
$[p]\neq [q]$ in $\B(Y)$.

Neither there are non-trivial mappings $T_X\colon \B_k(X)\to \A_k(X)$
such that  
\begin{equation}\label{studs}
\begin{array}[c]{ccc}
\B_k(X) & \stackrel{B_X}{\longrightarrow} & \widehat H^{k,k}(X) \\
\downarrow \scriptstyle{T_X} & &  \downarrow \scriptstyle{Id} \\
\A_k(X) & \stackrel{A_X}{\longrightarrow} &   \widehat H^{k,k}(X)
\end{array}
\end{equation}
commutes and $T_X( c_1(L)\w\mu)=c_1(L)\cap T_X \mu$ for each line bundle $L$.  
Just take $X$ that has a nontrivial line bundle with flat metric
and a meromorphic non-trivial section, as in the following example.

\begin{ex} Let $X$ be a complex $1$-dimensional torus. It is well-known that
two different points $p_1$ and $p_2$ are not rationally equivalent, i.e.,
there is no meromorphic function whose divisor is $[D]:=[p_1]-[p_2]$.
But the cohomology class determined by $[D]$ is zero.
Let $L$ be the line bundle $\Ok(D)$ equipped with some Hermitian metric.
By the Poincar\'e-Lelong formula,
$\hat c_1(L)$ is a representative of the cohomology class of $[D]$ and is thus $d$-exact.
Hence, since $\hat c_1(L)$ is smooth,
the $dd^c$-lemma shows that there is a smooth global function $\phi$ such that $dd^c\phi=\hat c_1(L)$. 
If we modify the metric on $L$ by $\exp(-\phi)$
we  have $\hat c_1(L)=0$.
\end{ex}

\def\listing#1#2#3{{\sc #1}\ {\it #2},\ #3}

\end{document}